\documentclass[11pt]{amsart}
\usepackage{amscd}
\usepackage[all]{xy}
\usepackage{graphicx}
\usepackage{amsmath}
\usepackage{amsfonts}
\usepackage{amssymb}
\usepackage{latexsym}
\usepackage{slashed}
\usepackage{soul}
\usepackage{comment}
\usepackage{rotating}

\usepackage{amsmath, latexsym, amssymb}
\numberwithin{equation}{section}
\theoremstyle{plain}
\newtheorem{lemma}{Lemma}[section]
\newtheorem{proposition}[lemma]{Proposition}
\newtheorem{theorem}[lemma]{Theorem}
\newtheorem{corollary}[lemma]{Corollary}

\theoremstyle{definition}
\newtheorem{definition}[lemma]{Definition}
\newtheorem{remark}[lemma]{Remark}
\newtheorem{example}[lemma]{Example}

\newcommand{\R}{{\mathbb R}}

\newcommand{\F}{{\mathbb F}}

\newcommand{\Z}{{\mathbb Z}}

\newcommand{\Aa}{{\mathcal A}}
\newcommand{\Bb}{{\mathcal B}}
\newcommand{\Cc}{{\mathcal C}}    
\newcommand{\Dd}{{\mathcal D}}

\newcommand{\Hh}{{\mathcal H}}
\newcommand{\Kk}{{\mathcal K}}

\newcommand{\Ll}{{\mathcal L}}    

\newcommand{\Qq}{{\mathcal Q}}

\newcommand{\Om}{{\Omega}}

\newcommand{\eps}{{\varepsilon}}

\newcommand{\la}{\langle}
\newcommand{\ra}{\rangle}

\newcommand{\spann}{\mbox{\rm span}}

\newcommand{\Isom}{\mbox{{\rm Isom}}}
\newcommand{\Diffeo}{\mbox{{\rm Diffeo}}}
\newcommand{\Transl}{\mbox{{\rm Transl}}}

\date{\today}
\title[Almost formality of  manifolds of low dimension]
{Almost formality of manifolds of low dimension}

\author[D. Fiorenza]{Domenico Fiorenza}
\address{
Dipartimento di Matematica ``Guido Castelnuovo'',
Universit\`a   di Roma ``La Sapienza",
Piazzale Aldo Moro 2, 00185 Roma, 
Italy}
\email{fiorenza@mat.uniroma1.it} 

\author[K. Kawai]{Kotaro Kawai}
\address{Gakushuin University, 1-5-1, Mejiro, Toshima, Tokyo, 171-8588, Japan}
\email{kkawai@math.gakushuin.ac.jp}

\author[H.V. L\^e]{H\^ong V\^an L\^e}
\address{Institute of Mathematics, Czech Academy of Science,
Zitna 25, 11567 Praha 1, Czech Republic}
\email{hvle@math.cas.cz}

\author[L. Schwachh\"ofer]{Lorenz Schwachh\"ofer}
\address{Faculty for Mathematik,
TU Dortmund University,
Vogelpothsweg 87, 44221 Dortmund, Germany}
\email{lschwach@math.tu-dortmund.de}

\thanks{The second named author is supported by JSPS KAKENHI Grant Numbers JP17K14181, and 
 the research of  the third named author is  supported by the GA\v CR-project 18-00496S and  RVO:67985840. The fourth named author was supported by the Deutsche Forschungsgemeinschaft by grant SCHW 893/5-1.}

\begin{document}
	
\abstract  In this  paper  we introduce the notion of Poincar\'e DGCAs of Hodge type, which is a subclass of Poincar\'e DGCAs  encompassing the de Rham algebras of closed 
orientable   manifolds. Then we introduce  the  notion of the small algebra  and the small quotient algebra of a  Poincar\'e  DGCA of Hodge type.
Using these concepts,
 we investigate  the  equivalence class  of $(r-1)$ connected $(r>1)$  Poincar\'e DGCAs of Hodge type. In particular, we show  that a $(r-1)$ connected Poincar\'e  DGCA of Hodge type $\Aa^\ast$ of dimension $n \le 5r-3$   is $A_\infty$-quasi-isomorphic  to  an $A_3$-algebra  and  prove that  the only obstruction
to the formality of $\Aa^\ast$ is a  distinguished  Harrison cohomology class  
$[\mu_3] \in {\mathsf{Harr}}^{3,-1} (H^*(\Aa^\ast), H^*(\Aa^\ast))$.  Moreover, the cohomology class  
$[\mu_3]$ and the DGCA isomorphism class of $H^*(\Aa^\ast)$ determine the $A_\infty$-quasi-isomorphism class of $\Aa^\ast$. This can be seen as a Harrison cohomology version of  the  Crowley-Nordstr\"om  results \cite{CN} on rational homotopy type  of  $(r-1)$-connected   $(r>1)$ closed manifolds of dimension up to $5r-3$.   We also   derive the   almost  formality of  closed  $G_2$-manifolds,  which have been discovered  recently  by  Chan-Karigiannis-Tsang  in \cite{CKT}, from our results and the
Cheeger-Gromoll splitting   theorem.
\endabstract


\maketitle

\section{Introduction}\label{sec:intr}

By the seminal work of Sullivan \cite{Sullivan} it is known that the equivalence class of the de Rham algebra of a closed simply connected manifold determines its real homotopy type. Thus, it is desirable to obtain invariants which measure the formality of the de Rham algebra,  and more generally identify the rational homotopy type of a manifold, and the present paper aims to describe such invariants under suitable highly enough connectedness assumptions. Assuming (as we shall always do) that the underlying manifold is closed, connected and oriented, the de Rham algebra is a {\em Poincar\'e differential graded commutative algebra}, meaning that there is a top homology class which induces a non-degenerate pairing on the cohomology algebra.

This work is inspired by two recent approaches to this question. The first is the article \cite{CKT} where it is shown that the de Rham algebra of a $G_2$-manifold (i.e., a closed $7$-manifold with a torsion-free $G_2$-structure) is {\em almost formal}, i.e., equivalent to a DGCA all of whose differentials vanish except in the middle degree. Second, in \cite{CN} a new invariant for a Poincar\'e DGCA is introduced, called the {\em Bianchi-Massey tensor}  
and it is shown that this tensor is the only obstruction to formality of $(r-1)$-connected $(r >1)$ closed
manifolds in dimension at most $5r-3$. In fact, the de Rham algebra of such a manifold is equivalent to one with non-vanishing differential in certain dimensions only, and for $r = 2$, \cite{CN} implies that any simply connected closed
$7$-manifold is almost formal in the sense of \cite{CKT}, i.e., the $G_2$-structure is not needed in the simply connected case.

Our approach is to apply the homotopy transfer theorem to associate to an equivalence class of Poincar\'e DGCAs of Hodge type an equivalence class of $A_\infty$-algebras. We show that a simply connected Poincar\'e DGCA of Hodge type is equivalent to a certain {\em finite dimensional }Poincar\'e DGCA of Hodge type, so that the associated $A_\infty$-structure is simpler to describe. For instance, in the case of a $(r-1)$-connected $(r>1)$ Poincar\'e DGCA of Hodge  type of dimension at most $5r-3$, the resulting $A_\infty$-algebra is an $A_3$-algebra and thus determined by its cohomology algebra $H^\ast(\Aa^\ast)$ and the distinguished Hochschild cohomology class  $[\mu_3] \in {\mathsf{Hoch}}^{3,-1} (H^\ast(\Aa^\ast), H^\ast(\Aa^\ast))\subseteq {\mathsf{Hoch}}^2 (H^\ast(\Aa^\ast), H^\ast(\Aa^\ast))$. Moreover, $\mu_3$ is actually a Harrison cocycle. Since in characteristic zero the natural morphism from Harrison cohomology to Hochschild cohomology of a DGCA is injective \cite{Barr1968}, this means that $[\mu_3]$ is equivalently a Harrison cohomology class.
As it is known that  two simply-connected spaces are rationally homotopy equivalent if and only  their de Rham complexes  are  $Comm_\infty$-quasi-isomorphic \cite[Theorem 27]{Valette2012}, \cite{Kadeishvili2009}, this implies that the rational homotopy type of a $(r-1)$-connected $(r> 1)$ closed
manifold $X$ of dimension at most $5r-3$ is completely determined by $H^\ast(X)$ and $[\mu_3] \in {\mathsf{Harr}}^{3,-1} (H^\ast(X), H^\ast(X))$. Relating this to the results in \cite{CN}, this means that $[\mu_3]$ is the Harrison cohomology interpretation of the Bianchi-Massey tensor. More precisely, the Bianchi-Massey tensor appears via the Harrison-to-Andr\'e-Quillen spectral sequence for  $\Aa^\ast$  \cite{Barr1968, Loday1992}.

The condition of non-simply connectedness may be weakened to the existence of a Riemannian metric on $M$ for which all harmonic $1$-forms are parallel. This is the case, for instance, if $M$ carries a metric of nonnegative Ricci curvature and, in particular, for $G_2$-manifolds as these are necessarily Ricci flat. Namely, by a  
generalization of the Cheeger-Gromoll splitting theorem (Theorem \ref{thm:Cheeger-Gromoll})  
it follows that in this case, the de Rham algebra is equivalent to a polynomial algebra $\Qq^\ast[t_1, \ldots, t_k]$ with $k = b^1(M)$ $|t_i| = 1$ where $\Qq^\ast$ is a Poincar\'e DGCA of lower degree
(Proposition  \ref{prop:reduce-b1}). In consequence, we show $7$-manifolds with nonnegative Ricci curvature and positive $b^1$ must be formal; in particular, this is the case for $G_2$-manifold whose holonomy is properly contained in $G_2$. Note that  the question of the formality of $G_2$-manifolds  is still open.

Our paper  is organized  as follows.  In  Section \ref{sec:hodge}  we 
 define and study Poincar\'e DGCA's admitting a Hodge type decomposition. 
  In  Section  \ref{sec:Poincare} we define the {\em small quotient  algebra} of such  DGCA's of Hodge type, a Poincar\'e DGCA equivalent to the original one, which is finite dimensional if  the DGCA is simply connected. 
In Section \ref{sec:ainfty}  we utilize the homotopy transfer theorem and related techniques to show that the rational homotopy type of a $(r-1)$ connected Poincar\'e  DGCA  $\Aa^\ast$ of dimension $n \le 5r-3$ is determined by $H^\ast(\Aa^\ast)$ and $[\mu_3] \in {\mathsf{Harr}}^{3,-1} (H^\ast(\Aa^\ast), H^\ast(\Aa^\ast))$. In   Section \ref{sec:h1form} we generalize the discussion from simply connected manifolds to those which admit a Riemannian metric all of whose harmonic $1$-forms are parallel.  
Finally, in Appendix \ref{appendix.massey} we   prove a few  results  on higher  Massey   products on connected  Poincar\'e DGCAs.  
 
 \

{\it Conventions and notations.} 

- In this paper we  consider only
DGCAs over a fixed  field  $\F$  of characteristic   0.

- For a DGCA $(\Aa^\ast, d)$ we denote  by $\Aa^\ast_d$ the   subspace of cocycles in $\Aa^\ast$  and by $H^\ast(\Aa^\ast)$  the cohomology of $(\Aa^\ast, d)$. The Betti numbers of $\Aa^\ast$ are defined as $b^k(\Aa^\ast) := \dim \Hh^k(\Aa^\ast)$ or simply by $b^k$. 

\section{Poincar\'e  DGCAs of Hodge   type}\label{sec:hodge}

In this  section we recall the definition of a Poincar\'e DGCA over $\F$ (Definition \ref{def:Poincare}, cf. \cite[Definition 2.7]{CN}).
Then we define Hodge type decompositions of Poincar\'e  DGCAs (Definition \ref{def:splitting-map})   $\Aa^\ast$, 
provide examples and investigate their properties.

\begin{definition}\label{def:Poincare}(cf. \cite[Definition 2.7]{CN})
\begin{enumerate}
\item
A {\em Poincar\'e-algebra of degree $n$} is a finite dimensional graded commutative algebra $H^\ast = \bigoplus_{k=0}^n H^k$ together with an element $\int \in (H^n)^\vee$, where the latter denotes the dual space of $H^n$, such that the pairing
\[
\langle \alpha^k,\beta^l \rangle := \begin{cases} \int \alpha^k\cdot \beta^l & \mbox{if $k+l=n$},\\ 0 & \mbox{else} \end{cases}
\]
is non-degenerate, i.e., $\langle \alpha, H^\ast\rangle = 0$ iff $\alpha = 0$.
\item
A {\em Poincar\'e-DGCA of degree $n$}  is a  DGCA 
$(\Aa^\ast, d)$ with $\Aa^\ast = \bigoplus_{k=0}^n \Aa^k$ whose cohomology ring $H^\ast := H^\ast(\Aa^\ast)$ is a Poincar\'e-algebra of degree $n$.
\end{enumerate}
\end{definition}

Note that in \cite[Definition 2.7]{CN}, the degree of a Poincar\'e-algebra is called the {\em dimension}, but as later we wish to consider the dimension of $\Aa^\ast$ as a graded vector space, the notion of degree seems more appropriate.

We begin our discussion by introducing some general terminology on pairings on a graded vector space.
A {\em graded vector space of degree $n$} is a vector space of the form $V^\ast = \bigoplus_{k=0}^n V^k$, and a bilinear pairing $\langle-,-\rangle$ on such a $V$ is called {\em graded symmetric}, if $\langle V^k, V^l\rangle = 0$ for $k+l \neq n$ or, equivalently, if the map $\langle-,-\rangle: V^\ast \otimes V^\ast \to \F$ has degree $-n$,  and factors through $\mathrm{Sym}^2(V)$, i.e., $\langle \alpha^k, \beta^l\rangle =  (-1)^{kl} \langle \beta^l, \alpha^k \rangle$ for any $\alpha^k\in V^k$ and any $\beta^l\in V^l$.
As usual, we define the orthogonal complement of a (graded) subspace $W^\ast \subseteq V^\ast$ as
\[
W^\ast_\perp := \{ \alpha \in V^\ast \mid \langle \alpha, W^\ast \rangle = 0\},
\]
and call $V^\ast_\perp$ the {\em null-space of $V^\ast$}. A subspace $W^\ast \subseteq V^\ast$ is called {\em non-degenerate } if $W^\ast \cap W^\ast_\perp = 0$. We call the pairing non-degenerate if $V^\ast$ is non-degenerate, i.e., iff $V^\ast_\perp = 0$. On the quotient $Q^\ast := V^\ast/V^\ast_\perp$, there is a unique non-degenerate pairing $\langle - , - \rangle_{Q^\ast}$ such that
\begin{equation} \label{eq:pairing-Q}
\langle \pi(\alpha), \pi(\beta)\rangle_{Q^\ast} = \langle \alpha, \beta \rangle,
\end{equation}
where $\pi: V^\ast \to Q^\ast$ is the canonical projection.

Let us now apply all of this to a Poincar\'e-DGCA $\Aa^\ast$ of degree $n$. The pairing $\langle-,-\rangle$ on $H^\ast(\Aa^\ast)$ induces a graded pairing, denoted by the same symbol, defined as
\begin{equation} \label{eq:pairing-A}
\langle \alpha^k,\beta^l\rangle = \begin{cases} \int [\alpha^k\cdot \beta^l] & \mbox{if $k+l=n$},\\ 0 & \mbox{else}. \end{cases}
\end{equation}
where $[\cdot]$ stands for the projection $\Aa^n = \Aa^n_d \to H^n(\Aa^\ast)$. It immediately follows that $\langle-,-\rangle$ satisfies
\begin{equation} \label{eq:properties-pairing}
\begin{array}{lll}
\langle \alpha^k, \beta^l\rangle & = & (-1)^{kl} \langle \beta^l, \alpha^k \rangle,\\
\langle \alpha^k \cdot \beta^l, \gamma^r\rangle & = & \langle \alpha^k, \beta^l \cdot \gamma^r\rangle,\\
\langle d\alpha^k, \beta^l \rangle & = & (-1)^{k+1} \langle \alpha^k, d\beta^l\rangle.
\end{array}
\end{equation}
 Poincar\'e-DGCAs  are examples of cyclic  $A_\infty$-algebras  that have been introduced  by Kontsevich \cite{Kontsevich1994}. The cyclic   DGCA $(\Om^\ast(M), d)$, $H^\ast_{dR}(M)$ have been considered  in  other papers, e.g. \cite{Hajek2018}.
\begin{definition} \label{def:splitting-map}
Let $(\Aa^\ast, \langle - , - \rangle)$ be a Poincar\'e-DGCA.
\begin{enumerate}
\item
We call $\Aa^\ast$ {\em non-degenerate }if the pairing $\langle - , - \rangle$ from (\ref{eq:pairing-A}) is non-degenerate.
\item
A {\em harmonic subspace} is a subspace $\Hh^\ast \subseteq \Aa^\ast_d$ complementary to $d\Aa^\ast$.
\item
A {\em splitting map of $\Aa^\ast$} is a splitting $\imath: H^\ast \to \Aa^\ast_d$ which splits the short exact sequence
\[\xymatrix{
0\ar[r] & d\Aa^{\ast-1} \ar[r] & \Aa_d^\ast \ar[r]_{pr} & H^\ast\ar@/_1pc/[l]_\imath \ar[r] & 0}.
\]
\item
A {\em Hodge type decomposition} is a direct sum decomposition of the form
\begin{equation} \label{eq:Hodge-type}
\Aa^\ast = d\Aa^{\ast-1} \oplus \Hh^\ast \oplus \Bb^\ast,
\end{equation}
where $\Hh^\ast$ is a harmonic subspace and such that
	\begin{equation}\label{def:Hodge-DGA2}
	\langle \Hh^\ast \oplus \Bb^\ast, \Bb^\ast\rangle = 0
	\end{equation}
\item
If $\Aa^\ast$ admits a Hodge type decomposition, then we call $\Aa^\ast$ 
a {\em Poincar\'e DGCA of Hodge type}. 
\end{enumerate}
\end{definition}

\begin{remark}\label{rem:nondegenerate}
 Evidently there is a one-to-one correspondence of harmonic subspaces $\Hh^\ast \subseteq \Aa^\ast_d$ and splitting maps of $\Aa^\ast$, as the image of a splitting map is a harmonic subspace. As $\imath: H^\ast \to \Hh^\ast$ preserves the pairing $\langle - , - \rangle$, it follows that the graded subspace $\Hh^\ast \subseteq \Aa^\ast$ is non-degenerate.
\end{remark}

\begin{remark}
By (\ref{eq:properties-pairing}), $\Aa^\ast_\perp \subseteq \Aa^\ast$ is a $d$-invariant ideal, whence there is an exact sequence of DGCA
\begin{equation} \label{eq:SES-Q}
0 \longrightarrow \Aa^\ast_\perp \longrightarrow \Aa^\ast \longrightarrow \Qq^\ast \longrightarrow 0,
\end{equation} 
where $\Qq^\ast$ is by definition the quotient $\Aa^\ast/\Aa^\ast_\perp$. 
\end{remark}

\begin{example} \label{ex:deRham}
The prototypical example of a Poincar\'e algebra of degree $n$ of Hodge type (which motivates our terminology) is the de Rham algebra $(\Omega^*(M),d)$ of a closed smooth oriented manifold $M$, with $\int$ given by the integration of $n$-forms. The Hodge decomposition w.r.t. some Riemannian metric $g$ on $M$
\begin{equation} \label{eq:Hodge-deRham}
\Om^\ast(M) = d\Om^{\ast-1}(M) \oplus \Hh^\ast(M) \oplus d^\ast \Om^{\ast+1}(M)
\end{equation}
yields a Hodge type decomposition in the sense of Definition \ref{def:splitting-map}, and the space $\Hh^\ast(M)$ of $\triangle_g$-harmonic forms is a harmonic subspace in the above sense. As we shall see in Remark \ref{rem:harm1} below, any other harmonic subspace $\Hh^\ast \subseteq \Om^\ast_d(M)$ will give rise to a Hodge type decomposition as well.
\end{example}

\begin{remark} \label{rem:harm1}
If $(\Aa^\ast, d)$ is a Poincar\'e-DGCA  admitting a Hodge-type decomposition,
 then every harmonic subspace $\Hh^\ast \subseteq \Aa^\ast_d$ may occur as summand of a Hodge-type decomposition (\ref{eq:Hodge-type}).
Namely, given a Hodge type decomposition (\ref{eq:Hodge-type}) and a harmonic subspace $\hat \Hh^\ast \subseteq \Aa^\ast_d$, from the identity $\Hh^\ast\oplus d\Aa^{\ast-1}= \hat{\Hh}^\ast\oplus d\Aa^{\ast-1}$ it  follows that there is a linear map $\alpha: \Hh^\ast \to d\Aa^{\ast-1}$ such that
\[
\hat \Hh^\ast = \{ v + \alpha(v) \mid v \in \Hh^\ast\}.
\]
Letting $\hat{\mathcal{B}}^\ast := \{ x - \alpha^\dagger(x) - \frac12 \alpha \alpha^\dagger(x) \mid x \in \mathcal{B}^\ast\}$, where $\alpha^\dagger: \mathcal{B}^\ast \to \Hh^\ast$ is the unique map satisfying $
\la \alpha^\dagger(x), v \ra = \la x, \alpha(v) \ra \quad \mbox{ for all $x \in \mathcal{B}^\ast$, $v \in \Hh^\ast$}$,
the decomposition $
\Aa^\ast = \hat \Hh^\ast \oplus d\Aa^{\ast-1} \oplus \hat{\mathcal{B}}^\ast$ 
is a Hodge type decomposition. 
\end{remark}

Given a Hodge-type decomposition (\ref{eq:Hodge-type}), it follows that the restriction $d: \Bb^\ast \to d\Aa^\ast$ is both injective and surjective and hence has an inverse $d^-: d\Aa^\ast \to \Bb^\ast$. We may extend $d^-$ to all of $\Aa^\ast$ by defining $d^-_{|\Hh^\ast \oplus \Bb^\ast} = 0$. Thus, $(d^-)^2 = 0$, and
\begin{equation} \label{eq:dd-d}
dd^-d = d, \qquad d^-dd^- = d^-.
\end{equation}
It follows that the projections in (\ref{eq:Hodge-type}) are given by
\begin{equation} \label{eq:projection}
pr_{\Hh^\ast} = 1 - [d, d^-], \qquad pr_{d\Aa^{\ast-1}} = dd^-, \qquad pr_{\Bb^\ast = d^-\Aa^{\ast+1}} = d^- d,
\end{equation}
where $[d, d^-] = dd^- + d^- d$ is the super-commutator. Therefore, (\ref{eq:Hodge-type}) may be written as
\begin{equation} \label{eq:Hodge-type-d-}
\Aa^\ast = d\Aa^{\ast-1} \oplus \Hh^\ast \oplus d^-\Aa^{\ast+1} = dd^-\Aa^\ast \oplus \Hh^\ast \oplus d^-d \Aa^\ast,
\end{equation}
and setting $\Aa^\ast_{d^-} := \ker d^- = \Hh^\ast \oplus d^-\Aa^\ast$, (\ref{def:Hodge-DGA2}) implies
	\begin{equation}\label{eq:Hodge-DGA3}
	\langle \Aa^\ast_d, d\Aa^{\ast-1} \rangle = \langle \Aa^\ast_{d^-}, d^-\Aa^{\ast+1} \rangle = 0.
	\end{equation}
We shall call elements in $\Aa^\ast_{d^-}$ and $d^-\Aa^{\ast+1}$ {\em $d^-$-closed } and {\em $d^-$-exact}, respectively.

Recall that a  DGA  over $\F$ is called \emph{connected} if $H^0(\Aa^\ast)\cong \mathbb{F}$ and {\em $r$-connected } if it is connected and $H^k(\Aa^\ast) = 0$ for $k = 1, \ldots, r$. A $1$-connected DGA  is also  called \emph{simply connected}.

\begin{remark}\label{rem:poincare}
\begin{enumerate}
\item
For $\Aa^\ast = (\Omega^*(M),d)$ for a closed oriented manifold $M$ with its Hodge decomposition w.r.t. a Riemannian metric $g$ (cf. Example \ref{ex:deRham}) the map $d^-$ in (\ref{eq:Hodge-type-d-}) is \emph{not} the codifferential $d^*$ in (\ref{eq:Hodge-deRham}). In particular, the supercommutator $\triangle_g=[d,d^*]$ is not a projector. Rather, it is not hard to see that $d^-$ and $d^\ast$ are related by the formula $d^\ast = \triangle_g d^{-}$.

\item
Note that $(\Omega^*(M),d)$ is connected iff $M$ is connected. However, the $r$-connectedness of $(\Omega^*(M),d)$, i.e., the vanishing of $b^k(M)$, $1 \leq k \leq r$, is a weaker notion than the $r$-connectedness of $M$, as the latter means that the homotopy groups $\pi_k(M)$ are trivial for $k \leq r$ (which is stronger than the vanishing of $b^k(M)$, $1 \leq k \leq r$).

\item
If the harmonic subspace $\Hh^\ast$ happens to be a subalgebra of $\Aa^\ast$, i.e., if the product of two harmonic elements is harmonic, then $(\Aa^\ast,d)$ is a formal DGCA.

\item
A Riemannian manifold $(M, g)$ for which the product of $\triangle_g$-harmonic forms is $\triangle_g$-harmonic is called {\em geometrically formal}.  It is known that compact symmetric spaces are geometrically formal \cite{Sullivan1975}, but in general, there are strong topological obstructions for $M$ to admit a geometrically formal Riemannian metric, see  \cite{Kotschick} for details and classification results of low dimensional geometrically formal Riemannian manifolds. The relation  between  the Hodge  decomposition
of  the  de Rham complex  on a  closed Riemannian  manifold $M$  and   the minimal model  of $M$ in the case $b^1 (M) =0$ has  been   proposed  by Sullivan \cite{Sullivan, Sullivan1975}.  The    Hodge decomposition  is also important  in the   proof  of the formality of  closed K\"ahler manifolds by Deligne-Griffiths-Morgan-Sullivan \cite{DGMS}.
	
\item
We will be only interested in connected  Poincar\'e DGCAs. 
In this case the product in cohomology $H^0(\Aa^\ast)\otimes H^k(\Aa^\ast)\to H^k(\Aa^\ast)$ is easily seen to be multiplication by scalars on the $\mathbb{F}$-vector space $H^k(\Aa^\ast)$.

\end{enumerate}
\end{remark}

\begin{lemma} \label{lem:N-triv-cohom}
If the Poincar\'e  DGCA $\Aa^\ast$ admits a Hodge type decomposition, then the differential ideal
$\Aa^\ast_\perp$ is invariant under $d^-$, and has the decomposition
\begin{equation} \label{eq:decomp-N}
\Aa^{\ast}_\perp  =  dd^-\Aa^{\ast}_\perp \oplus d^-d\Aa^{\ast}_\perp.
\end{equation}
In particular, $(\Aa^{\ast}_\perp, d)$ 
has trivial cohomology.
\end{lemma}

\begin{proof}
We begin by showing that the restriction of the harmonic projector $pr_{{\Hh}^\ast}$ to $\Aa^{\ast}_\perp$is zero. To see this, let  $\alpha^k \in \Aa^{k}_\perp$ 
and write
\[
\alpha^k =pr_{{\Hh}^\ast}(\alpha^k)+dd^{-}\alpha^k+d^-d\alpha^k
\]
Then, for any $\beta^{n-k}\in \Hh^{n-k}$, we have
\[
0=\langle \alpha^k, \beta^{n-k}\rangle= \langle pr_{{\Hh}^\ast}(\alpha^k),\beta^{n-k}\rangle,
\]
as $\Hh^{n-k}\subseteq \Aa^{n-k}$ and by the orthogonality relations (\ref{def:Hodge-DGA2})-(\ref{eq:Hodge-DGA3}). Since the pairing $\langle-,-\rangle$ is nondegenerate on $\Hh^\ast$ (Remark \ref{rem:nondegenerate}), this implies $pr_{{\Hh}^\ast}(\alpha^k)=0$. Now we can show that $\Aa^{\ast}_\perp$ is stable with respect to both projections $dd^{-}$ and $d^{-}d$. As we have shown that
\[
\alpha^k =dd^{-}\alpha^k+d^-d\alpha^k
\]
for any $\alpha^k \in \Aa^{k}_\perp$, it will suffice to show that $d^-d\alpha^k\in \Aa^{k}_\perp$. By the orthogonality relations (\ref{def:Hodge-DGA2}) and by the decomposition (\ref{eq:Hodge-type}), we only need to show that $\langle d^-d\alpha^k, dd^-\beta^{n-k}\rangle=0$. We compute
\begin{align*}
\langle d^-d\alpha^k, dd^-\beta^{n-k}\rangle&=(-1)^{k+1}\langle dd^-d\alpha^k, d^-\beta^{n-k}\rangle\\
&=(-1)^{k+1}\langle d\alpha^k, d^-\beta^{n-k}\rangle\\
&=\langle \alpha^k, dd^-\beta^{n-k}\rangle=0
\end{align*}
where we used $dd^-d=d$ and the fact that $dd^-\beta^{n-k}\in \Aa^{n-k}$ 
by (\ref{eq:dd-d}-\ref{eq:projection}).
This shows that we have the direct sum decomposition
\[
\Aa^{\ast}_\perp  =  dd^-\Aa^{\ast}_\perp \oplus d^-d\Aa^{\ast}_\perp.
\]
Finally, we show that $d^{-}\Aa^{k}_\perp\subseteq \Aa^{k-1}_\perp$. To see this, let $\alpha^k\in \Aa^{k}_\perp$. As $d^{-}\Aa^{k}_\perp=d^{-}dd^{-}\Aa^{k}_\perp$, we want to show that $\langle d^{-}dd^{-}\alpha^k,\beta^{n-k+1}\rangle=0$ for any $\beta^{n-k+1}\in \Aa^{n-k+1}$. As $d^{-}dd^{-}\alpha^k\in d^-\Aa^{k-1}$, by the orthogonality relations (\ref{def:Hodge-DGA2}) and from the direct sum decomposition (\ref{eq:Hodge-type}) we see that it is not restrictive to assume $\beta^{n-k+1}=dd^-\gamma^{n-k+1}$, with $\gamma^{n-k+1}\in \Aa^{n-k+1}$. We have
\[
\langle d^{-}dd^{-}\alpha^k,dd^-\gamma^{n-k+1}\rangle=(-1)^k\langle dd^{-}dd^{-}\alpha^k,d^-\gamma^{n-k+1}\rangle=0
\]
where we used that $\Aa^{*}_\perp$ is stable with respect to $dd^{-}$. As $\mathrm{Id}_{\Aa^{\ast}_\perp}=[d,d^-]$
 with $d^{-}$ a degree -1 endomorphism of $\Aa^{\ast}_\perp$, we see that the identity of $\Aa^{\ast}_\perp$ is null-homotopic and so $(\Aa^{\ast}_\perp,d)$ is an acyclic complex.
\end{proof}

\begin{corollary} \label{cor:pi-QI}
If $\Aa^\ast$ is of Hodge type, then the projection $\pi: \Aa^\ast \to \Qq^\ast$, where $\Qq^\ast=\Aa^\ast/\Aa^\ast_\perp$, is a quasi-isomorphism. 
\end{corollary}

\begin{remark}
As the induced pairing on $\Qq^\ast$ is nondegenerate by construction, Corollary \ref{cor:pi-QI} implies that every Hodge type Poincar\'e DGCA is equivalent to a nondegenerate one.
This is actually true even without the assumption that the Poincar\'e-DGCA $\Aa^\ast$ is of Hodge type, but the general proof is much more involved: see \cite[Theorem 1.1]{LS2007}.
\end{remark}

\begin{example}\label{ex:not-Hodge} We conclude this section with an example of a Poincar\'e DGCA admitting no Hodge type decomposition. Let $\Aa^\ast$ be the graded vector space on linear generators $\{x_0, x_2, x_3, x_4\}$ with $\deg(x_i)=i$. That is $\Aa^k$ is 1-dimensional for $k\in\{0,2,3,4\}$ and 0-dimensional for any other value of $k$. Introduce a DGCA structure on $\Aa^\ast$ by declaring the product and the differential on the generators as follows:
\begin{itemize}
\item the only non-zero products are $x_0 \cdot x_i = x_i\cdot x_0= x_i$ and $x_2 \cdot x_2 = x_4$;
\item the only non-zero differential is $dx_2 = x_3$.
\end{itemize}
The cohomology of $\Aa^\ast$ is the graded algebra on linear generators $[x_0]$ and $[x_4]$ with the products $[x_0]\cdot[x_0]=[x_0]$, $[x_0]\cdot[x_4]=[x_4]\cdot[x_0]=[x_4]$ and $[x_4]\cdot[x_4]=0$. Therefore, the only harmonic subspace is $\Hh^\ast = \spann(x_0, x_4)$. The linear functional $\int\colon H^4(\Aa^\ast)\to \mathbb{F}$ given by $\int[x_4]=1$ makes  $\Aa^\ast$ a  degree 4 Poincar\'e DGCA.  The $d$-invariant ideal $(\Aa_\perp^\ast, d)$ consists of $\Aa^3$ in degree 3 with the zero differential 
and so it is manifestly not acyclic. Therefore, $(\Aa^\ast, \langle-,-\rangle)$ cannot have a Hodge type decompositions by Lemma \ref{lem:N-triv-cohom}. 
\end{example}
\begin{remark}
Observe that in Example \ref{ex:not-Hodge} $\Hh^\ast$ is a subalgebra of $\Aa^\ast$, whence the inclusion $\Hh^\ast \hookrightarrow \Aa^\ast$ is a quasi-isomorphism, and evidently $\Hh^\ast$ is of Hodge type. That is, being of Hodge type is a property not preserved by quasi-isomorphisms.
\end{remark}

\begin{remark}\label{rem:CFL} In \cite[Lemma 11.1]{CFL2015} Cieliebak-Fukaya-Latschev proved  that  a finite
dimensional non-degenerate Poincare DGCA 
always admits a Hodge type decomposition.
\end{remark}

\section{Small quotient algebras  of  Poincar\'e   DGCAs of Hodge type}\label{sec:Poincare}

In this section we define the small algebra
of a Poincar\'e  DGCA of Hodge type $\Aa^\ast$  and the small  quotient algebra  of $\Aa^*$ (Definition \ref{def:small}).
 These algebras are DGCAs equivalent to $\Aa^\ast$. As a result we show that any simply connected Poincar\'e-DGCA of Hodge type is equivalent to a non-degenerate finite dimensional Poincar\'e-DGCA  (Corollary \ref{cor:equiv-findim}).
We investigate   the   equivalence class of  a $(r-1)$ connected ($r>1$)  Poincar\'e DGCA (Theorem \ref{thm:dim-finite}, Corollary  \ref{cor:formal-lowdim})  and  compare  our results  with some known  results (Remark \ref{rem:Miller}).

\begin{definition}
Let $\Aa^\ast$ be a Poincar\'e DGCA with a Hodge-type decomposition (\ref{eq:Hodge-type-d-}).  A $\Hh^*$-\emph{subalgebra} is a DG-subalgebra of $\Aa^\ast$ which is $d^{-}$-invariant and contains $\Hh^*$.
\end{definition}

For any $\Hh^\ast$-subalgebra $\Cc^\ast$ of $\Aa^\ast$, the restriction of the inner product of $\Aa^\ast$ to $\Cc^\ast$ makes $\Cc^\ast$ a degree $n$ Poincar\'e DGCA and the Hodge-type decomposition of $\Aa^\ast$ induces a Hodge-type decomposition
\begin{align}
\Cc^\ast & = \Hh^\ast \oplus d\Cc^\ast \oplus d^-\Cc^\ast
\label{eq:Hodge-Aa}
\end{align}
as $\Cc^*$ contains $\Hh^*$ and is invariant under the projections (\ref{eq:projection}).
In particular, the inclusion $(\Cc^\ast, d) \hookrightarrow (\Aa^\ast, d)$ is a quasi-isomorphism. Taking the quotient $\Qq^\ast(\Cc^\ast) := \Cc^\ast/\Cc^\ast_\perp$, Corollary \ref{cor:pi-QI} and Lemma \ref{lem:N-triv-cohom} imply that $\Qq^\ast(\Cc^\ast)$ is a non-degenerate Poincar\'e algebra of Hodge type, and the inclusion and projection maps
\[
\Aa^\ast \longleftarrow \Cc^\ast \longrightarrow \Qq^\ast(\Cc^\ast)
\]
are quasi-isomorphisms, so that $\Aa^\ast$ is equivalent to $\Qq^\ast(\Cc^\ast)$.

\begin{definition} \label{def:small} Let $\Aa^\ast = d\Aa^\ast \oplus \Hh \oplus d^-\Aa^\ast$ be a Poincar\'e DGCA with a Hodge type decomposition. The {\em small algebra of $\Aa^\ast$}, denoted by $\Aa^*_{\mathrm{small}}$ or $\Aa^*_{\mathrm{small}}(\Hh^*)$ is the (unique) smallest $\Hh^*$-subalgebra of $\Aa^*$. The {\em small quotient algebra of $\Aa^\ast$} is the quotient DGCA $\Qq^*_{\mathrm{small}}(\Hh^*)$ or $\Qq^*_{\mathrm{small}} := \Qq^\ast(\Aa^*_{\mathrm{small}}) = \Aa^*_{\mathrm{small}}/(\Aa^*_{\mathrm{small}})_\perp$.
\end{definition}

Clearly, the smallest algebra is well defined as $\Aa^\ast$ is a $\Hh^*$-subalgebras and the intersection of an arbitrary family of $\Hh^*$-subalgebras is a $\Hh^*$-subalgebra, so that $\Aa^*_{\mathrm{small}}$ is the intersection of all $\Hh^*$-subalgebras of $\Aa^*$. We shall, however, give a more accessible description of $\Aa^*_{\mathrm{small}}$.

For  a $\Hh^*$-subalgebra  $\Cc^*$ of $\Aa^*$  we denote by $\Cc^{i,j}\subseteq \Cc^{i+j}$ the image of the multiplication $\Cc^i\otimes \Cc^j\to \Cc^{i+j}$.

\begin{proposition}\label{prop:minimal}
Let $\Aa^\ast$ be a simply connected Poincar\'e DGCA endowed with a Hodge type decomposition. Then its small algebra $\Aa^\ast_{\mathrm{small}}$ is defined by the recursive formula
\begin{equation}\label{eq:recursion1}
\begin{cases}
&\Aa^0_{\mathrm{small}}=\F\cdot 1_{\Aa},\\
& \Aa^1_{\mathrm{small}} = 0,\\
&\Aa^k_{\mathrm{small}}  =  \Hh^k\oplus\; dd^- \spann \{\Aa^{l_1,l_2}_{\mathrm{small}}, l_1, l_2 \geq 2, l_1 + l_2 = k\}\\
 & \qquad\qquad \oplus\; d^- \spann \{ \Aa^{l_1,l_2}_{\mathrm{small}}, l_1, l_2 \geq 2, l_1 + l_2 = k+1\};\qquad  k \geq 2
\end{cases}
\end{equation}
In particular $\Aa^\ast_{\mathrm{small}}$ is finite dimensional.
\end{proposition}

\begin{proof}
Let us denote by $\Cc^\ast$ the algebra defined by the recursion (\ref{eq:recursion1}). Thus, our aim is to show that $\Cc^\ast = \Aa^\ast_{\mathrm{small}}$.

It is immediate from the definition of $\Cc^\ast$ and from the identity $d^-dd^-=d^-$ that $\Hh^\ast \subseteq \Cc^\ast$ and $d\Cc^k\subseteq \Cc^{k+1}$ and $d^-\Cc^k\subseteq \Cc^{k-1}$.
To see that $\Cc^\ast$ is closed under multiplication, we need to show that, if $\alpha^k\in \Cc^k$ and $\beta^l\in \Cc^l$, then $\alpha^k\cdot \beta^l\in \Cc^{k+l}$. If $k\leq 1$ or $l\leq 1$ there is nothing to be proven. So let us assume $k,l\geq 2$. By equations (\ref{eq:Hodge-Aa}),(\ref{eq:projection}) we have
\[
\alpha^k \cdot \beta^l = pr_{\Hh^\ast}(\alpha^k \cdot \beta^l) + dd^-(\alpha^k \cdot \beta^l) + d^- (d\alpha^k \cdot \beta^l) + (-1)^k d^-(\alpha^k \cdot d\beta^l),
\]
and the right-hand side manifestly belongs to $\Cc^{k+l}$, as $\Cc^\ast$ is $d$-closed so that $d\alpha^k, d\beta^l \in \Cc^\ast$.

Conversely, it follows by induction of $k$ that $\Cc^k \subseteq \Aa^k_{\mathrm{small}}$. Indeed, for $k = 0,1$ this is obvious, and if $k \geq 2$, then it is evident that any $\Hh^\ast$-subalgebra containing $\Cc^l$ for $l < k$ also must contain $\Cc^k$ by (\ref{eq:recursion1}).
\end{proof}

\begin{remark}
It is worth noticing that the proof of Proposition \ref{prop:minimal} does not use the existence of a pairing on $\Aa^\ast$ (and so in particular the existence of a Hodge-type decomposition) nor the graded commutativity of the multiplication. That is, Proposition \ref{prop:minimal} actually shows that a simply connected
DGA $\Aa^\ast$ concentrated in degrees $[0,n]$ and with finite dimensional cohomology is always equivalent to a finite-dimensional DGA $\Aa^\ast_{\mathrm{small}}$. 
\end{remark}

As $\Qq_{\text{small}}$ is equivalent to $\Aa^\ast$, we have thus shown the following.

\begin{corollary} \label{cor:equiv-findim}
Let $\Aa^\ast$ be a simply connected Poincar\'e-DGCA of degree $n$ and of Hodge type. Then $\Aa^\ast$ is equivalent to a finite dimensional non-degenerate Poincar\'e-DGCA of degree $n$ and of Hodge type.
\end{corollary}

\begin{remark}
The result of Corollary \ref{cor:equiv-findim} holds more generally without the assumption that the Poincar\'e-DGCA $\Aa^\ast$ is of Hodge type. Namely, one can apply Theorem 1.1 from \cite{LS2007} to the small algebra $\Aa^\ast_{\mathrm{small}}$ to get a finite dimensional non-degenerate Poincar\'e-DGCA $\widetilde{\Aa^\ast_{\mathrm{small}}}$ of degree $n$. By Remark  \ref{rem:CFL}, this algebra will admit a Hodge type decomposition.
\end{remark}

Given a  connected Poincar\'e DGCA of degree $n$ $\Aa^\ast$, with cohomology algebra $H^\ast := H^\ast(\Aa^\ast)$, let $H^\ast_+ := \bigoplus_{k=1}^n H^k$.

\begin{definition} \label{def-generating}
A {\em generating subspace} is a graded subspace $H^\ast_{gen} \subseteq H^\ast_+$ complementary to $H^\ast_+ \cdot H^\ast_+$, i.e., there is a direct sum decomposition
\begin{equation} \label{eq:H*gen}
H^\ast_+ = H^\ast_{gen} \oplus (H^\ast_+ \cdot H^\ast_+).
\end{equation}
\end{definition}

As $H^\ast$ is finite dimensional, $H^\ast_{gen}$ always exists and $H^1 \subseteq H^\ast_{gen}$. Moreover, if $\Aa^\ast$ is $(r-1)$-connected, so that $H^k = 0$ for $k = 1, \ldots, r-1$, then $H^k \subseteq H^\ast_{gen}$ for $0 < k \leq 2r-1$. Induction on the degree easily implies that each element in $H^\ast_+ \cdot H^\ast_+$ is spanned by products of elements in $H^\ast_{gen}$. That is, $H^\ast_{gen}$ generates $H^\ast_+$ as an algebra.

Let $S^\ast(H^\ast_{gen})$ be the graded symmetric tensor algebra of elements of $H^\ast_{gen}$. Multiplication induces a morphism of graded algebras $S^\ast(H^\ast_{gen}) \xrightarrow{\cdot} H^\ast$, which is surjective, as $H^\ast_{gen}$ generates $H^\ast$ and $H^0 = \F$ is spanned by the image of the identity in $S^\ast(H^\ast_{gen})$. Therefore, we obtain the exact sequence
\begin{equation} \label{eq:short-sequ1}
0 \longrightarrow K^\ast \longrightarrow S^\ast(H^\ast_{gen}) \xrightarrow{\quad\cdot\quad} H^\ast \longrightarrow 0
\end{equation}
where the kernel of the multiplication $K^\ast \subseteq S^\ast(H^\ast_{gen})$ is an ideal. In fact, decomposing $S^\ast(H^\ast_{gen}) = \F \oplus H^\ast_{gen} \oplus S^\ast_{\geq2}(H^\ast_{gen})$, 
where $S^\ast_{\geq k}(H^\ast) \subseteq S^\ast(H^\ast)$ denotes the ideal of graded polynomials of degree at least $k$, multiplication maps the first two summands isomorphically to $H^0$ and $H^\ast_{gen}$, respectively, and $S^\ast_{\geq2}(H^\ast_{gen})$ to $H_+ \cdot H_+$. Therefore, we may restrict (\ref{eq:short-sequ1}) to a subsequence
\begin{equation} \label{eq:short-sequ2}
0 \longrightarrow K^\ast \longrightarrow S_{\geq2}^\ast(H^\ast_{gen}) \xrightarrow{\quad\cdot\quad} H_+^\ast \cdot H_+^\ast \longrightarrow 0.
\end{equation}

A linear map $\imath_0: H^\ast_{gen} \to \Aa^\ast_d$ will be called a {\em partial splitting map} if $pr \circ \imath_0$ is the inclusion of  $H^\ast_{gen}$ into $H^\ast(\Aa^\ast)$ with  $pr: \Aa^\ast_d \to H^\ast$ from Definition \ref{def:splitting-map}.
As $H^\ast_{gen}$ is finite dimensional, a partial splitting map always exists, and we denote its image by $\Hh^\ast_{gen} \subseteq \Aa^\ast_d$.
Then $\imath_0$ induces an algebra morphism $S^\ast(H^\ast_{gen}) \to \Aa^\ast_d$ which by abuse of notation we also denote by $\imath_0$, and 
 $pr \circ \imath_0: S^\ast(H^\ast_{gen}) \to H^\ast$ coincides with the multiplication map `` $\cdot$ " from above. Observe that $\imath_0(S^\ast(H^\ast_{gen}))$ is the algebra generated by $\Hh^\ast_{gen} \subseteq \Aa^\ast_d$.

\begin{definition}\label{def:adapted}
Let $H^\ast_{gen} \subseteq H^\ast_+$ be a generating subspace and $\imath_0: H^\ast_{gen} \to \Aa^\ast_d$ be a partial splitting map.
We call a splitting map $\imath: H^\ast(\Aa^\ast)\to \Aa^\ast_d$ an {\em extension of $\imath_0$} if $\imath_{|H^\ast_{gen}} = \imath_0$ and $\imath(H^\ast(\Aa^\ast))  \subseteq \imath_0(S^\ast(H^\ast_{gen}))$. In this case, we say that the harmonic subspace $\Hh^\ast=\imath(H^\ast(\Aa^\ast))$ is {\em $\imath_0$-adapted}; we call $\Hh^\ast_{gen}$ a {\em generating harmonic space}, and $\Hh^\ast = \imath(H^\ast)$ an {\em extension of $\Hh^\ast_{gen}$}.
\end{definition}

Thus, we have the following commuting diagrams of short exact sequences 
\begin{equation} \label{eq:exact-i0}
\xymatrix{
0\ar[r] & K^\ast\ar[d]_{\imath_0}\ar[r] & S^\ast(H^\ast_{gen})\ar[d]_{\imath_0}\ar[r]^\cdot & H^\ast \ar@{=}[d]\ar[r] & 0\\
0\ar[r] & \Kk^\ast\ar[r] & \imath_0(S^\ast(H^\ast_{gen}))\ar[r]_{pr} & H^\ast\ar@/_1pc/[l]_\imath \ar[r] & 0
}
\end{equation}
\begin{equation} \label{eq:exact-i1}
\xymatrix{
0\ar[r] & K^\ast\ar[d]_{\imath_0}\ar[r] & S_{\geq 2}^\ast(H^\ast_{gen})\ar[d]_{\imath_0}\ar[r]^\cdot & H_+^\ast \cdot H^\ast_+ \ar@{=}[d]\ar[r] & 0\\
0\ar[r] & \Kk^\ast\ar[r] & \imath_0(S_{\geq 2}^\ast(H^\ast_{gen}))\ar[r]_{pr} & H_+^\ast \cdot H^\ast_+\ar@/_1pc/[l]_{\imath_1} \ar[r] & 0
}
\end{equation}
where the top rows are the sequence (\ref{eq:short-sequ1}) and (\ref{eq:short-sequ2}), respectively, and the vertical maps are surjective, and where $\Kk^\ast := \imath_0(K^\ast) \subseteq \Aa^\ast$.

\begin{lemma} \label{lem:i0-adapted}
There is a one-to-one correspondence between $\imath_0$-adapted splitting maps $\imath: H^\ast \to \Aa^\ast_d$ and splitting maps $\imath_1: H^\ast_+ \cdot H^\ast_+ \to \imath_0(S_{\geq 2}^\ast(H^\ast_{gen}))$ of the exact sequence in the bottom row of (\ref{eq:exact-i1}). In particular, such a splitting map always exists.
\end{lemma}

\begin{proof}
Given a splitting map $\imath_1: H^\ast_+ \cdot H^\ast_+ \to \imath_0(S_{\geq 2}^\ast(H^\ast_{gen}))$ of (\ref{eq:short-sequ2}), it follows that $\imath := \imath_0 \oplus \imath_1: H^\ast = H^\ast_{gen} \oplus H_+^\ast \cdot H^\ast_+ \to \Aa^\ast_d$ (using the decomposition (\ref{eq:H*gen})) is a splitting of (\ref{eq:short-sequ1}) whose image is contained in $\imath_0(S^\ast(H^\ast_{gen}))$. Conversely, given a splitting $\imath: H^\ast \to \imath_0(S^\ast(H^\ast_{gen}))$ of (\ref{eq:short-sequ1}), its restriction to $H_+^\ast \cdot H^\ast_+$ maps to $\imath_0(S_{\geq 2}^\ast(H^\ast_{gen}))$ and hence yields a splitting map of (\ref{eq:short-sequ2}). The existence of such a splitting follows as $H_+^\ast \cdot H^\ast_+$  is finite dimensional.
\end{proof}

By definition, if $\Aa^\ast$ is $(r-1)$-connected, then $H^\ast_+ \subseteq H^\ast_{\geq r}$ and hence, $H^\ast_+ \cdot H^\ast_+ \subseteq H^\ast_{\geq 2r}$. Thus, if $\Hh^\ast \subseteq \Aa^\ast_d$ is an $\imath_0$-adapted harmonic subspace  
then by (\ref{eq:H*gen}) it immediately follows that
\begin{equation} \label{eq:Hk-gen}
\begin{array}{llll}
H^k = H^k_{gen} & \mbox{for $1 \leq k \leq 2r - 1$}, & \ & H^{2r} = H^{2r}_{gen} \oplus H^r \cdot H^r,\\
\Hh^k = \Hh^k_{gen} & \mbox{for $1 \leq k \leq 2r - 1$}, & & \Hh^{2r} \subseteq \Hh^{2r}_{gen} \oplus \mu(\Hh^r \odot \Hh^r),
\end{array}
\end{equation}
where $\mu$ is the multiplication map in $\Aa^\ast$. In particular, as $\Aa^\ast$ is $(r-1)$-connected, it follows that both $H^k_{gen}$ and $\Hh^k_{gen}$ vanish for $1 \leq k \leq r-1$. Moreover, most of the spaces of degree $\leq 2r$ in (\ref{eq:exact-i1}) vanish because of (\ref{eq:Hk-gen}), whence we immediately conclude
\begin{equation} \label{eq:K-gen}
\begin{array}{llll}
K^k & = \Kk^k = 0 \qquad \mbox{for $1 \leq k \leq 2r - 1$},\\
K^{2r} & = \ker (\cdot: H^r \odot H^r \to H^{2r})\\
\Kk^{2r} & = \mu(\Hh^r \odot \Hh^r) \cap d\Aa^{2r-1}.
\end{array}
\end{equation}

Note that any subalgebra of $\Aa^\ast$ which contains $\Hh^\ast$ and hence $\Hh^\ast_{gen} = \imath_0(H^\ast_{gen})$ must also contain $\imath_0(S^\ast(H^\ast_{gen})) = \Hh^\ast \oplus \Kk^\ast$, using the splitting of the bottom exact sequence in (\ref{eq:exact-i1}). This holds, in particular, for the small algebra $\Aa^\ast_{\mathrm{small}}$. Comparing this with the recursive description of $\Aa^\ast_{\mathrm{small}}$ in (\ref{eq:recursion1}), this together with (\ref{eq:K-gen}) yields

\begin{equation}\label{eq:recursion2}
\begin{cases}
\Aa^0_{\mathrm{small}}=\Hh^0=\F\cdot 1_{\Aa},\\
\Aa^k_{\mathrm{small}} = \Hh^k= 0, & 1\leq k\leq r-1\\
\Aa^k_{\mathrm{small}}  =  \Hh^k, & r\leq k\leq 2r-2\\
\Aa^{2r-1}_{\mathrm{small}}  =  \Hh^{2r-1} \oplus d^-\Kk^{2r};\\
\Aa^{2r}_{\mathrm{small}}  =  \Hh^{2r}\oplus\Kk^{2r}\oplus\; d^- \Aa^{r,r+1}_{\mathrm{small}}.\\
 \end{cases}
\end{equation}

Thus, when passing to the quotient algebra $ \Qq^k_{\mathrm{small}}$, we can summarize our discussion so far as follows.

\begin{theorem} \label{thm:dim-finite}
Let $\Aa^\ast$ be a $(r-1)$ connected ($r>1$)  Poincar\'e-DGCA of degree $n$  of Hodge type. Then $\Aa^\ast$ is equivalent to a finite dimensional non-degenerate Poincar\'e-DGCA $\Qq^\ast_{\mathrm{small}}$ of Hodge type, 
satisfying
\begin{itemize}
\item
$\Qq^k_{\mathrm{small}}= \Hh^k$ for $0\leq k\leq 2r-2$ and for $n-2r+2\leq k\leq n$.
\item
if the multiplication map $\cdot: H^r \odot H^r \to H^{2r}$ is injective, then $\Qq^k_{\mathrm{small}}= \Hh^k$ also for $k = 2r-1$ and $n-2r+1$.
\end{itemize}
\end{theorem}

\begin{proof} This is an immediate consequence of (\ref{eq:recursion2}) as $\Hh^k \subseteq \Qq^k_{\mathrm{small}}$ and, as it is a quotient, $\dim \Qq^k_{\mathrm{small}} \leq \dim \Aa^k_{\mathrm{small}}$. Also, $\dim \Qq^k_{\mathrm{small}} = \dim \Qq^{n-k}_{\mathrm{small}}$ as $\Qq^\ast_{\mathrm{small}}$ is a non-degenerate Poincar\'e DGCA.

The last statement follows as the injectivity of this multiplication map implies that $K^{2r} = 0$ by (\ref{eq:K-gen}), whence $\Kk^{2r} = \imath_0(K^{2r}) = 0$, so that $\Aa^{2r-1}_{\mathrm{small}}  =  \Hh^{2r-1}$ by (\ref{eq:recursion2}).
\end{proof}

\begin{corollary} \label{cor:formal-lowdim}
Let $\Aa^\ast$ be a $(r-1)$-connected ($r > 1$) Poincar\'e-DGCA of Hodge type. 
Then $\Aa^\ast$ is equivalent to a finite dimensional non-degenerate Poincar\'e DGCA $\Qq^*_{\mathrm{small}}$ for which the differential $d\colon \Qq^{k-1}_{\mathrm{small}}\to \Qq^k_{\mathrm{small}}$ is possibly nonzero only for $2r\leq k\leq n-2r+1$. In particular
\begin{enumerate}
\item
If $n\leq 4r-2$, then $\Aa^\ast$ is formal.
\item
If $n =4r-1$, then $d: \Qq^{k-1}_{\mathrm{small}} \to \Qq^k_{\mathrm{small}}$ vanishes for $k \neq 2r$.
\end{enumerate}
\end{corollary}

Of course, this applies to $\Aa^\ast = \Om^\ast(M)$ where $M$ is a closed oriented connected $n$-dimensional manifold with $b^k(M) = 0$, $1 \leq k \leq r-1$.

\begin{remark}\label{rem:Miller} 
\begin{enumerate}
\item
Corollary \ref{cor:formal-lowdim} (1) for DGCAs   associated with closed (orientable) manifolds has  been  first proved  by  Miller in \cite{Miller1979} using Quillen's functor, including the statement that any closed simply connected manifold of dimension $\leq 6$ is formal. It  can be also derived from Crowley-Nordstr\"om's  result \cite[Theorem 1.3]{CN} based on  their invented Bianchi-Massey tensor.
\item
If $M$ is a simply connected $7$-manifold, then Corollary \ref{cor:formal-lowdim} implies that $M$ is {\em almost formal} in the sense of \cite{CKT}. That is, $\Om^\ast(M)$ is equivalent to a DGCA $\Qq^\ast$ whose only possibly non-vanishing differential is $d: \Qq^3 \to \Qq^4$. This almost-formality was shown in \cite{CKT} for $G_2$-manifolds, but as our result shows, the $G_2$-structure is not needed.
\end{enumerate}
\end{remark}

\section{$A_\infty$-quasi-isomorphism type  of highly connected   Poincar\'e DGCAs of Hodge type}\label{sec:ainfty}

In this section, using  homotopy transfer theorem and related technique, we   prove that
a $(r-1)$ connected  Poincare  DGCA $\Aa^\ast$ of Hodge type of dimension  $n \le 5r-3$ is $A_\infty$-quasi-isomorphic   to an $A_3$-algebra (Theorem \ref{thm:A2-algebra}),
that the only obstruction to the formality of $\Aa^\ast$  is  the cohomology class
$[\mu_3] \in  {\mathsf{Hoch}}^2 (H^*(\Aa^\ast), H^*(\Aa^*))$, where $\mu_3$ is the class of the ternary multiplication  in the $A_3$-algebra (Theorem \ref{thm:obstruction}), and that the cohomology class  
$[\mu_3]$ and the DGCA isomorphism class of $H^*(\Aa^\ast)$ determine the $A_\infty$-quasi-isomorphism class of $\Aa^\ast$ (Theorem   \ref{thm:detects}). 
Then  we show the relation between $\mu_3$  and the Massey product (Remark \ref{rem:Massey}, Corollary \ref{cor:Massey-b1})  and compare our result  with a closely related result  by  Crowley-Nordstr\"om (Remark \ref{rem:compareCN}).

\begin{theorem}\label{thm:A2-algebra}
Let $\Aa^\ast$ be a $(r-1)$-connected $(r>1)$ Poincar\'e DGCA of degree $n$, endowed with a Hodge-type decomposition whose harmonic subspace $\Hh^\ast$ is adapted to a generating subspace $H^\ast_{gen}$. If $n\leq 5r-3$ then $\Hh^\ast$ carries an $A_3$-algebra structure\footnote{i.e., an $A_\infty$ algebra structure with vanishing multiplications $m_k$ for $k\geq 4$} (with trivial differential) making it $A_\infty$-quasi-isomorphic to $\Aa^\ast$. Moreover, if $n\leq 4r-2$ then the ternary multiplication of this $A_3$-algebra vanishes, so in particular $\Aa^\ast$ is formal.
\end{theorem}
\begin{proof}
 Let us therefore focus on the case $4r-1\leq n\leq 5r-3$. As we know that the smallest $\Hh^\ast$-quotient  $\Qq^\ast_{\mathrm{small}}$ of $\Aa^\ast$ is  equivalent 
to $\Aa^*$ by Corollary \ref{cor:formal-lowdim}, 
we only need to prove that $\Hh^\ast$ carries an $A_3$-algebra structure with trivial differential making it quasi-isomorphic to $\Qq^\ast_{\mathrm{small}}$. We can show this by means of the homotopy transfer theorem \cite{Kadeishvili1980, Merkulov}.
\par
By Poincar\'e duality and equation (\ref{eq:recursion2})  $\Qq^\ast_{\mathrm{small}}$ is given by
\begin{equation}\label{eq:recursion5}
\begin{cases}
&\Qq^0_{\mathrm{small}}=\Hh^0\cong \F\\
& \Qq^k_{\mathrm{small}} = \Hh^k= 0, \qquad\qquad\qquad\qquad 1\leq k\leq r-1\\
&\Qq^k_{\mathrm{small}}  =   \Hh^k, \qquad\qquad\qquad\qquad r\leq k\leq 2r-2\\
 &\Qq^{k}_{\mathrm{small}}  =  \Hh^{k} \oplus \Ll^k, \qquad\qquad\qquad\qquad 2r-1 \leq k\leq n-2r+1\\ 
 &\Qq^{k}_{\mathrm{small}}  =   \Hh^{k}, \qquad\qquad\qquad\qquad n-2r+2 \leq k\leq n-r\\
& \Qq^{k}_{\mathrm{small}} = \Hh^{k}= 0, \qquad\qquad\qquad\qquad n-r+1\leq k\leq n-1\\
 &\Qq^{n}_{\mathrm{small}}=\Hh^{n}\cong\F.
\end{cases}
\end{equation}
where $\Ll^k=dd^-\Qq^{k}_{\mathrm{small}} \oplus d^-d\Qq^{k}_{\mathrm{small}}$ for every $k$ and, in particular
\[
\Ll^{2r-1}=d^- \mu(\Hh^r\odot \Hh^r).
\]
Notice  that  $d^-$ identically vanishes on $\Qq^{k}_{\mathrm{small}}$ for $k\leq 2r-1$ and for $k\geq n-2r+2$. By the homotopy transfer theorem, $\Hh^\ast$ carries an $A_\infty$-algebra structure making it $A_\infty$-quasi-isomorphic to  $\Qq^\ast_{\mathrm{small}}$. To prove the statement in the theorem we therefore only need to show that this $A_\infty$-algebra structure is actually an $A_3$-
algebra structure, i.e., that all the multiplications $m_k$ vanish for $k\geq 4$. One has a convenient tree summation formula to express the higher multiplications $m_k$ obtained by homotopy transfer, see \cite{KS2001}. Namely, $m_k$ can be expressed as a sum over rooted trivalent trees with $k$ leaves. Each tail edge of such
a tree is decorated by inclusion $j\colon \Hh^\ast\hookrightarrow \Qq^\ast_{\mathrm{small}}$, each internal edge is decorated by the operator $d^-\colon \Qq^\ast_{\mathrm{small}}\to \Qq^{\ast-1}_{\mathrm{small}}$ and
the root edge is decorated by the operator $\pi_{\Hh^\ast}\colon  \Qq^\ast_{\mathrm{small}}\to \Hh^\ast$; every
internal vertex is decorated by the multiplication $\mu$ in $\Qq^\ast_{\mathrm{small}}$. 
\par
In order to get a nonzero operation, we can not have a subgraph of the form
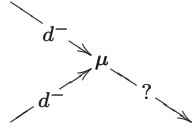
\[
\begin{xy}
,(-24,0);(-19.2,3.2)*{\,\,\,\scriptstyle{d^-}\,}**\dir{-}
,(-19.2,3.2)*{\,\,\,\scriptstyle{d^-}\,};
(-12,8)*{\,\,\scriptstyle{\mu}\,\,}**\dir{-}?>*\dir{>}
,(-12,8)*{\,\,\scriptstyle{\mu}\,\,};(-6,4)*{\,\scriptstyle{?}\,}**\dir{-}
,(-6,4)*{\,\scriptstyle{?}\,};
(0,0)**\dir{-}?>*\dir{>}
,(-24,16);(-18,12)*{\,\scriptstyle{d^-}\,}**\dir{-}
,(-18,12)*{\,\scriptstyle{d^-}\,};
(-12,8)*{\,\,\scriptstyle{\mu}\,\,}**\dir{-}?>*\dir{>}
\end{xy}
\]
where ? can be either $d^-$ or $\pi_{\Hh^\ast}$.
Namely, in order to get a possibly nonzero contribution from this graph, its homogeneous entries $a_{k_1}$ and $b_{k_2}$ should be in $\Qq^{k_1}_{\mathrm{small}}$ and $\Qq^{k_2}_{\mathrm{small}}$ with $k_1,k_2\geq 2r$. The term $d^-a_{k_1}\cdot d^-b_{k_2}$ will then be in $\Qq^{k}_{\mathrm{small}}$ with $k\geq 4r-2$, where $k=k_1+k_2-2$. As $n\leq 5r-3$ we have $4r-2\geq n-r+1$, and so $k\geq n-r+1$. Therefore, $d^-a_{k_1}\cdot d^-b_{k_2}$ is surely zero unless $k_1+k_2-2=n$. For $k_1+k_2-2=n$ we can compute
\[
\int d^-a_{k_1}\cdot d^-b_{k_2}=\langle d^-a_{k_1}, d^-b_{k_2}\rangle = 0,
\]
by the orthogonality relation (\ref{eq:Hodge-DGA3}). As $\int\colon \Qq^{n}_{\mathrm{small}}\to \mathbb{R}$ is an isomorphism, this gives $d^-a_{k_1}\cdot d^-b_{k_2}=0$.
\\
Also, we can not have a subgraph of the form
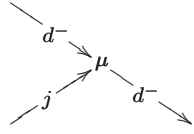
\[
\begin{xy}
,(-24,0);(-19.2,3.2)*{\,\scriptstyle{j}\,}**\dir{-}
,(-19.2,3.2)*{\,\scriptstyle{j}\,};
(-12,8)*{\,\,\scriptstyle{\mu}\,\,}**\dir{-}?>*\dir{>}
,(-12,8)*{\,\,\scriptstyle{\mu}\,\,};(-6,4)*{\,\scriptstyle{d^-}\,}**\dir{-}
,(-6,4)*{\,\scriptstyle{d^-}\,};
(0,0)**\dir{-}?>*\dir{>}
,(-24,16);(-18,12)*{\,\scriptstyle{d^-}\,}**\dir{-}
,(-18,12)*{\,\scriptstyle{d^-}\,};
(-12,8)*{\,\,\scriptstyle{\mu}\,\,}**\dir{-}?>*\dir{>}
\end{xy}
\]
Namely, if  its homogeneous entries are $a_{k_1}$ and $b_{k_2}$ (from top to bottom), then we need to have $k_1\geq 2r$ and $k_1+k_2-1\leq n-2r+1$ in order to have a possibly nonzero contribution. The two inequalities together give $k_2\leq n-4r+2\leq r-1$. So the only possibility for having a nonzero contribution is $k_2=0$. As $\Qq^{0}_{\mathrm{small}}=\mathbb{F}$ acting as scalars on $\Qq^{\ast}_{\mathrm{small}}$ via the multiplication $\mu$, we have
\[
d^-(d^-a_{k_1}\cdot b_0)=b_0(d^-)^2a_{k_1}=0.
\]
In conclusion, the only two graphs actually appearing in the tree summation formula are (absorbing the permutations of the tree branches in the graded commutativity of the multiplication) are
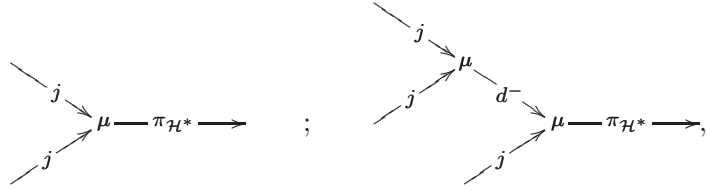
\[
\begin{xy}
,(-12,-8);(-7.2,-4.8)*{\,\scriptstyle{j}\,}**\dir{-}
,(-7.2,-4.8)*{\,\scriptstyle{j}\,};
(0,0)*{\,\,\scriptstyle{\mu}\,}**\dir{-}?>*\dir{>}
,(-12,8);(-6,4)*{\,\scriptstyle{j}\,}**\dir{-}
,(-6,4)*{\,\scriptstyle{j}\,};
(0,0)*{\,\,\scriptstyle{\mu}\,}**\dir{-}?>*\dir{>}
,(0,0)*{\,\,\scriptstyle{\mu}\,};
(9.6,0)*{\,\scriptstyle{\pi_{\Hh^\ast}}\,}**\dir{-}
,(9.6,0)*{\,\scriptstyle{\pi_{\Hh^\ast}}\,};
(19.2,0)**\dir{-}?>*\dir{>}
\end{xy}\qquad;
\qquad
\begin{xy}
,(-12,-8);(-7.2,-4.8)*{\,\scriptstyle{j}\,}**\dir{-}
,(-7.2,-4.8)*{\,\scriptstyle{j}\,};
(0,0)*{\,\,\scriptstyle{\mu}\,}**\dir{-}?>*\dir{>}
,(-24,0);(-19.2,3.2)*{\,\scriptstyle{j}\,}**\dir{-}
,(-19.2,3.2)*{\,\scriptstyle{j}\,};
(-12,8)*{\,\,\scriptstyle{\mu}\,\,}**\dir{-}?>*\dir{>}
,(-12,8)*{\,\,\scriptstyle{\mu}\,\,};(-6,4)*{\,\scriptstyle{d^-}\,}**\dir{-}
,(-6,4)*{\,\scriptstyle{d^-}\,};
(0,0)*{\,\,\scriptstyle{\mu}\,}**\dir{-}?>*\dir{>}
,(-24,16);(-18,12)*{\,\scriptstyle{j}\,}**\dir{-}
,(-18,12)*{\,\scriptstyle{j}\,};
(-12,8)*{\,\,\scriptstyle{\mu}\,\,}**\dir{-}?>*\dir{>}
,(0,0)*{\,\,\scriptstyle{\mu}\,};
(9.6,0)*{\,\scriptstyle{\pi_{\Hh^\ast}}\,}**\dir{-}
,(9.6,0)*{\,\scriptstyle{\pi_{\Hh^\ast}}\,};
(19.2,0)**\dir{-}?>*\dir{>}
\end{xy},
\]
defining the multiplications $m_2\colon \Hh^{k_1}\otimes \Hh^{k_2}\to \Hh^{k_1+k_2}$ and $m_3\colon \Hh^{k_1}\otimes \Hh^{k_2}\otimes \Hh^{k_3}\to \Hh^{k_1+k_2+k_2-1}$, respectively.
\\
If $n\leq 4r-2$, we can verbatim repeat the above argument, but now all the $\Ll^k$;s are zero so also the multiplication $m_3$ vanishes. The $A_3$-algebra structure on $\Hh^\ast$ therefore becomes a DGCA structure with trivial differential, and we recover the formality of $\Aa^\ast$ in this case (Corollary \ref{cor:formal-lowdim}).
\end{proof}
\begin{remark}\label{rem:expilicy-m3}
The multiplications $m_2$ and $m_3$ of the $A_3$ algebra structure on $\Hh^\ast$ described in the proof of Theorem \ref{thm:A2-algebra} are explicitly given by 
\begin{align}
m_2(\alpha,\beta)&=\pi_{\Hh^\ast}(\alpha\cdot\beta)\label{eq:mu2}\\
m_3(\alpha,\beta,\gamma)&=\pi_{\Hh^\ast}(d^-(\alpha\cdot \beta)\cdot\gamma-(-1)^{\deg \alpha} \alpha\cdot d^-(\beta\cdot\gamma))\label{eq:mu3}
\end{align}
see \cite[Theorem 3.4]{Merkulov}.
Moreover, as the pairing $\langle-,-\rangle$ on $\Hh^\ast$ is nondegenerate, the ternary multiplication $m_3$ is completely encoded in the tensor
\[
\tau(\alpha,\beta,\gamma,\delta)=\langle m_3(\alpha,\beta,\gamma),\delta\rangle. 
\] 
\end{remark}

\begin{remark}\label{rem:1-gives-0}
It is immediate to see that
\[
m_3(\alpha,\beta,\gamma)=0 \qquad \text{if}\quad \min\{\deg\alpha,\deg\beta,\deg\gamma\}=0.
\]
Indeed, the operator $d^-$ vanishes on $\Hh^\ast$, the projection $\pi_{\Hh^\ast}$ vanishes on the image of $d^-$ and the elements in $ \Hh^0\subseteq \Aa^0$ act as scalars via the multiplication in $\Aa^\ast$. From this one sees that the expression for $m_3$ given in Remark \ref{rem:expilicy-m3} identically vanishes when one of the arguments of $m_3$ has degree zero.
\end{remark}

\begin{corollary}\label{cor:towards-hoch}
Let $\Aa^\ast$ a $(r-1)$-connected ($r>1)$ Poincar\'e DGCA of degree $n$ of Hodge type with an adapted choice $\imath\colon H^*(\Aa^\ast)\hookrightarrow \Aa^\ast_d$ of harmonic forms. If $n\leq 5r-3$, then $H^\ast(\Aa^\ast)$ carries an $A_3$-algebra structure with zero differential, whose binary multiplication is the  the multiplication induced by $\Aa^\ast$, while the ternary multiplication is
\[
\mu_3([\alpha],[\beta],[\gamma])=[m_3(\imath[\alpha],\imath[\beta],\imath[\gamma])]
\]
making it $A_\infty$-quasi-isomorphic to $\Aa^\ast$. Moreover, if $n\leq 4r-2$ then $\mu_3$ vanishes.
\end{corollary}
\begin{proof}
The isomorphism $\imath\colon H^\ast(\Aa^\ast)\xrightarrow{\sim} \Hh^\ast$ with inverse $\pi_{\Hh^\ast}\colon \Hh^\ast\xrightarrow{\sim} H^\ast(\Aa^\ast)$ transfers the $A_3$-algebra structure on $\Hh^\ast$ from Theorem  \ref{thm:A2-algebra} to an $A_3$-algebra structure with the same properties on $H^\ast(\Aa^\ast)$. So the only thing we need to prove is the identity
\[
[m_2(\imath[\alpha],\imath[\beta])]=[\alpha]\cdot[\beta].
\]
This is straightforward, as 
\[
[\alpha]\cdot[\beta]=[(\imath[\alpha])\cdot(\imath[\beta])]=[\pi_{\Hh^\ast}((\imath[\alpha])\cdot(\imath[\beta]))]=[m_2(\imath[\alpha],\imath[\beta])].
\]
\end{proof}

\begin{lemma}
The degree -1 trilinear map $\mu_3\colon H^{k_1}(\Aa^\ast)\otimes H^{k_2}(\Aa^\ast))\otimes H^{k_3}(\Aa^\ast)\to H^{k_1+k_2+k_3-1}(\Aa^\ast)$  from Corollary \ref{cor:towards-hoch} is a cocycle in the Hochschild complex of the DGCA $H^\ast(\Aa^\ast)$, and so it defines a Hochschild cohomology class $[\mu_3]\in \mathsf{Hoch}^{3,-1}(H^\ast(\Aa^\ast),H^\ast(\Aa^\ast))\subseteq  \mathsf{Hoch}^{2}(H^\ast(\Aa^\ast),H^\ast(\Aa^\ast))$.
\end{lemma}
\begin{proof}
Immediate from Corollary \ref{cor:towards-hoch}. see, e.g., \cite{Lunts2007} for details.
\end{proof}

\begin{theorem}\label{thm:obstruction}
Let $\Aa^\ast$ a $(r-1)$-connected ($r>1)$ Poincar\'e DGCA of degree $n$ of Hodge type  with $n\leq 5r-3$. The Hochschild cohomology class $[\mu_3]$ is independent of the choice of adapted harmonic forms and of associated Hodge decomposition, and so it defines a distinguished element
\[
[\mu_3^{\Aa^\ast}]\in \mathsf{Hoch}^2(H^\ast(\Aa^\ast),H^\ast(\Aa^\ast))
\]
depending only on the Poincar\'e DGCA $\Aa^\ast$. Moreover the class $[\mu_3^{\Aa^\ast}]$ is the only obstruction to the formality of $\Aa^\ast$. 
\end{theorem}
\begin{proof}
It follows by a general result by Kadeishvili \cite{Kadeishvili1988} and Kaledin \cite{Kaledin2005} that the only obstruction to the formality of an $A_3$-algebra with trivial differential is the Hochschild cohomology class of the trilinear multiplication, see \cite[Corollary 5.7]{Lunts2007}. Therefore, as  $(H^\ast(\Aa^\ast),0,\cdot, \mu_3)$ and $(\Aa^\ast,d,\cdot)$ are $A_\infty$-quasi-isomorphic, $[\mu_3]$ is the only obstruction to the formality of $\Aa^\ast$. To conclude, we need to show that the class $[\mu_3]$ is actually independent of the choice of adapted harmonic forms and of associated Hodge decomposition. To see this, let $\tilde{\iota}$ and $\tilde{d}^-$ be the operators occurring in a different choice, and let  $\tilde{\mu}_3$ be the associated the trilinear multiplication. Then we will have $A_\infty$-quasi-isomorphism
\[
\imath_\infty\colon (H^\ast(\Aa^\ast),0,\cdot, \mu_3) \to (\Aa^\ast,d,\cdot); \qquad\tilde{\imath}_\infty\colon (H^\ast(\Aa^\ast),0,\cdot, \tilde{\mu}_3) \to (\Aa^\ast,d,\cdot)
\]
extending the linear morphisms $\imath$ and $\tilde{\imath}$, respectively. By considering an $A_\infty$-up to homotopy inverse quasi-isomorphism $\tilde{\pi}_\infty \colon (\Aa^\ast,d,\cdot)\to (H^\ast(\Aa^\ast),0,\cdot, \tilde{\mu}_3)$ we therefore get a homotopy commutative triangle of $A_\infty$-quasiisomorphisms
\[
\xymatrix{
&(\Aa^\ast,d,\cdot)\\
 (H^\ast(\Aa^\ast),0,\cdot, \mu_3) \ar[ru]^{\imath_\infty}\ar[rr]^{\phi_\infty}&&(H^\ast(\Aa^\ast),0,\cdot, \tilde{\mu}_3)\ar[lu]_{\tilde{\imath}_\infty}
}
\]
for some $A_\infty$-quasi-isomorphism $\phi_\infty$. Passing to cohomology we get the commutative diagram
\[
\xymatrix{
&H^\ast(\Aa^\ast)\\
 H^\ast(\Aa^\ast) \ar[ru]^{H^\ast(\imath)=\mathrm{id}}\ar[rr]^{H^\ast(\phi_1)}&&H^\ast(\Aa^\ast)\ar[lu]_{H^\ast(\tilde{\imath})=\mathrm{id}}
},
\]
where $\phi_1$ is the linear component of $\phi_\infty$. But then $\phi_1=\mathrm{id}\colon H^\ast(\Aa^\ast)\to H^\ast(\Aa^\ast)$, and so 
\[
\phi_\infty\colon (H^\ast(\Aa^\ast),0,\cdot, \mu_3)\to (H^\ast(\Aa^\ast),0,\cdot, \tilde{\mu}_3)
\]
is an $A_\infty$-quasi-isomorphism whose linear component is the identity. Spelling out the trilinear component in the definition of $A_\infty$-morphism, i.e.,
\[
\sum_{r+s+t=3 \atop r,t\geq 0, s\geq 1}(-1)^{rs+t}\phi_{r+1+t}(\mathrm{id}^{\otimes r}\otimes \mu_s\otimes \mathrm{id}^{\otimes t})=\sum_{j\in\{1,2,3\} \atop  {i_1+\cdots+i_j=3\atop i_1,\dots,i_j\geq 1}}(-1)^{u}
\tilde{\mu}_j(\phi_{i_1}\otimes\cdots\otimes\phi_{i_j}),
\]
where $u=\sum_{k=1}^{j-1}(j-k)(i_k-1)$, 
and using the fact $\phi_1=\mathrm{id}$, and that on both sides the differentials $\mu_1$ and $\tilde{\mu}_1$ are zero and the binary multiplications $\mu_2$ and $\tilde{\mu}_2$ are the multiplication $m$ in $H^\ast(\Aa^\ast)$, one finds 
\[
\mu_3-\phi_{2}(m\otimes \mathrm{id})+\phi_{2}(\mathrm{id}\otimes m)
=
m(\mathrm{id}\otimes\phi_{2})-
m(\phi_{2}\otimes\mathrm{id})
+
\tilde{\mu}_3,
\]
i.e., precisely,
\[
\tilde{\mu}_3=\mu_3+d_{\mathsf{Hoch}}\phi_2,
\]
and so $[\tilde{\mu}_3]=[\mu_3]$. 
\end{proof}
The same argument used in the proof of Theorem \ref{thm:obstruction} shows that the Hochschild class $[\mu_3^{\Aa^\ast}]$ detects the $A_\infty$-quasi-isomorphism class of an highly connected  Poincar\'e DGCA $\Aa^\ast$. The precise statement is the following.
\begin{theorem}\label{thm:detects}
Let $\Aa^\ast$ a $(r-1)$-connected ($r>1)$ Poincar\'e DGCA  of Hodge type of degree $n$ with $n\leq 5r-3$. The $A_\infty$-quasi-isomorphism class of $\Aa^\ast$ is completely encoded into the DGCA $(H^\ast(\Aa^\ast),\cdot)$ and in the Hochschild class $[\mu_3^{\Aa^\ast}] \in \mathsf{Hoch}^2(H^\ast(\Aa^\ast),H^\ast(\Aa^\ast))$.
\end{theorem}
\begin{proof}
We need to show that two $(r-1)$-connected Poincar\'e DGCA's of Hodge type of degree at most $5r-3$, $\Aa^\ast_1$ and $\Aa^\ast_2$  are $A_\infty$-quasi-isomorphic if and only if there exists a DGCA isomorphism $\phi\colon H^\ast(\Aa^\ast_1)\to H^\ast(\Aa^\ast_2)$ such that $[\phi^{-1}\circ(\mu_3^{\Aa^\ast_2}(\phi\otimes\phi\otimes\phi))]=[\mu_3^{\Aa^\ast_1}]$. In one direction, if $\phi_\infty\colon \Aa^\ast_1\to \Aa^\ast_2$ is an $A_\infty$-quasi-isomorphism, then the morphism $\phi\colon H^\ast(\Aa^\ast_1)\to H^\ast(\Aa^\ast_2)$ induced by the linear part $\phi_1$ of $\phi_\infty$ is an isomorphism of DGCA's (with trivial differential). The fact that $\phi_\infty$ is an $A_\infty$-morphism then gives, in particular, the identity
\[
\phi(\mu_3^{\Aa_1^\ast})={\mu}_3^{\Aa_2^\ast}(\phi\otimes\phi\otimes\phi)+\phi_{2}(\mu_2^{\Aa^\ast_1}\otimes \mathrm{id})-\phi_{2}(\mathrm{id}\otimes \mu_2^{\Aa^\ast_1})
+
\mu_2^{\Aa^\ast_2}(\phi\otimes\phi_{2})-
\mu_2^{\Aa^\ast_2}(\phi_{2}\otimes\phi),
\]
where $\mu_2^{\Aa^\ast_i}$ is the multiplication in $H^\ast(\Aa^\ast_i)$. By applying $\phi^{-1}$ on both sides and using that $\phi$ is a DGCA isomorphism, we find
\begin{align*}
\mu_3^{\Aa_1^\ast}=\phi^{-1}({\mu}_3^{\Aa_2^\ast}(\phi\otimes\phi\otimes\phi))&+(\phi^{-1}\circ \phi_{2})(\mu_2^{\Aa^\ast_1}\otimes \mathrm{id})-(\phi^{-1}\circ \phi_{2})(\mathrm{id}\otimes \mu_2^{\Aa^\ast_1})\\
&
+
\mu_2^{\Aa^\ast_1}(\mathrm{id}\otimes(\phi^{-1}\circ \phi_{2}))-
\mu_2^{\Aa^\ast_1}((\phi^{-1}\circ \phi_{2})\otimes\mathrm{id}),
\end{align*}
and so $[\mu_3^{\Aa^\ast_1}]=[\phi^{-1}\circ(\mu_3^{\Aa^\ast_2}(\phi\otimes\phi\otimes\phi))]$. Vice versa, assume we are given a DGCA isomorphism  $\phi\colon H^\ast(\Aa^\ast_1)\to H^\ast(\Aa^\ast_2)$ such that $[\phi^{-1}\circ(\mu_3^{\Aa^\ast_2}(\phi\otimes\phi\otimes\phi))]=[\mu_3^{\Aa^\ast_1}]$. Then, by definition of Hochschild cohomology, there exists a bilinear morphism
\[
\psi\colon H^\ast(\Aa_1^\ast)\otimes H^\ast(\Aa_1^\ast) \to H^{\ast-1}(\Aa_1^\ast)
\]
such that
\begin{align*}
\mu_3^{\Aa_1^\ast}=\phi^{-1}({\mu}_3^{\Aa_2^\ast}(\phi\otimes\phi\otimes\phi))&+\psi(\mu_2^{\Aa^\ast_1}\otimes \mathrm{id})-\psi(\mathrm{id}\otimes \mu_2^{\Aa^\ast_1})\\
&
+
\mu_2^{\Aa^\ast_1}(\mathrm{id}\otimes\psi)-
\mu_2^{\Aa^\ast_1}(\psi\otimes\mathrm{id}).
\end{align*}
Setting $\phi_1=\phi$ and $\phi_2=\phi\circ\psi$, the above identity precisely says that
\[
(\phi_1,\phi_2)\colon (H^\ast(\Aa_1^\ast),\mu_2^{\Aa^\ast_1},\mu_3^{\Aa^\ast_1})\to (H^\ast(\Aa_2^\ast),\mu_2^{\Aa^\ast_2},\mu_3^{\Aa^\ast_3}).
\]
is an $A_\infty$-isomorphism of $A_3$-algebras with trivial differential. As we have $A_\infty$-quasi-isomorphisms $\Aa_1^\ast\cong (H^\ast(\Aa_1^\ast),\mu_2^{\Aa^\ast_1},\mu_3^{\Aa^\ast_1})$ and $\Aa_2^\ast\cong (H^\ast(\Aa_2^\ast),\mu_2^{\Aa^\ast_2},\mu_3^{\Aa^\ast_2})$, this concludes the proof. 
\end{proof}

\begin{remark}[Massey products]\label{rem:Massey}
The ternary multiplication $\mu_3([\alpha],[\beta],[\gamma])$ is particularly interesting when both the products $[\alpha]\cdot[\beta]$ and $[\beta]\cdot[\gamma]$ vanish. Indeed, in this case we have
\[
\begin{cases}
\iota[\alpha]\cdot\iota[\beta]=d\eta\\
\iota[\alpha]\cdot\iota[\beta]=d\theta
\end{cases}
\]
for suitable elements $\eta,\theta\in d^-\Aa^\ast$ and so
\begin{align*}
\mu_3([\alpha],[\beta],[\gamma])&=[\pi_{\Hh^\ast}(\eta\cdot\iota[\gamma]-(-1)^{\deg \alpha} \iota[\alpha]\cdot \theta)]\\
&=[\eta\cdot\iota[\gamma]-(-1)^{\deg \alpha} \iota[\alpha]\cdot \theta],
\end{align*}
as $\eta\cdot\iota[\gamma]-(-1)^{\deg \alpha} \iota[\alpha]\cdot \theta$ is $d$-closed. This equation precisely states that $\mu_3([\alpha],[\beta],[\gamma])$ is a representative for the triple Massey product of $[\alpha]$, $[\beta]$ and $[\gamma]$. In particular, we see that if $\mu_3$ vanishes, then all triple Massey products vanish. One should however be warned that the converse is not true: there are examples of non-formal algebras with vanishing Massey products. See \cite{BMFM2018, Valette2012} for a detailed description of the relation between the $A_\infty$-algebra structure on $H^\ast(\Aa^\ast)$ obtained by homotopy transfer from the DGCA structure on a DGCA $\Aa^\ast$ and the (triple and higher) Massey products on $H^\ast(\Aa^\ast)$.
\end{remark}

\begin{remark}
Although an $A_3$-algebra can generally have nontrivial quadruple Massey products (see \cite[Theorem D]{BMFM2018}), simple degree considerations show that if $\Aa^\ast$ is a $(r-1)$-connected $(r>1)$ Poincar\'e DGCA of degree $n\leq 5r-3$ then all of the $l$ th order Massey products are trivial for $l \geq 4$. We show this in Appendix \ref{appendix.massey}.
\end{remark}

\begin{lemma}
The degree -1 trilinear map $\mu_3\colon H^{k_1}(\Aa^\ast)\otimes H^{k_2}(\Aa^\ast) \otimes H^{k_3}(\Aa^\ast)\to H^{k_1+k_2+k_3-1}(\Aa^\ast)$  from Corollary \ref{cor:towards-hoch} is a 2-cocycle in the Harrison subcomplex $\mathsf{Harr}^*(H^\ast(\Aa^\ast),H^\ast(\Aa^\ast))\subseteq \mathsf{Hoch}^*(H^\ast(\Aa^\ast),H^\ast(\Aa^\ast))$. In particular it defines an element in $\mathsf{Harr}^{3,-1}(H^\ast(\Aa^\ast),H^\ast(\Aa^\ast))\subseteq \mathsf{Harr}^2(H^\ast(\Aa^\ast),H^\ast(\Aa^\ast))$.
\end{lemma}
\begin{proof}
By definition of the Harrison complex \cite{Harrison}, we have to show that
\[
\mu_3([\alpha],[\beta],[\gamma])-(-1)^{\overline{[\alpha]}\cdot\overline{[\beta]}}\mu_3([\beta],[\alpha],[\gamma])+\mu_3([\beta],[\gamma],[\alpha])=0
\]
for any $[\alpha],[\beta],[\gamma]$ in $H^\ast(\Aa^\ast)$, where we have written $\overline{[x]}$ for the degree of the homogeneous element $[x]\in H^\ast(\Aa^\ast)$. We compute
\begin{align*}
\mu_3&([\alpha],[\beta],[\gamma])-(-1)^{\overline{[\alpha]}\cdot\overline{[\beta]}}\mu_3([\beta],[\alpha],[\gamma])+(-1)^{\overline{[\alpha]}\cdot(\overline{[\beta]}+\overline{[\gamma]})}\mu_3([\beta],[\gamma],[\alpha])=\\
&=
[m_3(\imath[\alpha],\imath[\beta],\imath[\gamma])]-(-1)^{\overline{[\alpha]}\cdot\overline{[\beta]}}[m_3(\imath[\beta],\imath[\alpha],\imath[\gamma])]+(-1)^{\overline{[\alpha]}\cdot(\overline{[\beta]}+\overline{[\gamma]})}[m_3(\imath[\beta],\imath[\gamma],\imath[\alpha])]\\
&=\pi_{\Hh^\ast}(d^-(\imath[\alpha]\cdot \imath[\beta])\cdot\imath[\gamma]-(-1)^{\overline{ [\alpha]}} \imath[\alpha]\cdot d^-(\imath[\beta]\cdot\imath[\gamma]))\\
&\qquad +\pi_{\Hh^\ast}(-(-1)^{\overline{[\alpha]}\cdot\overline{[\beta]}}d^-(\imath[\beta]\cdot \imath[\alpha])\cdot\imath[\gamma]+(-1)^{(\overline{[\alpha]}+1)\overline{ [\beta]}} \imath[\beta]\cdot d^-(\imath[\alpha]\cdot\imath[\gamma]))\\
&\qquad +\pi_{\Hh^\ast}((-1)^{\overline{[\alpha]}\cdot(\overline{[\beta]}+\overline{[\gamma]})}d^-(\imath[\beta]\cdot \imath[\gamma])\cdot\imath[\alpha]-(-1)^{(\overline{[\alpha]}+1)\overline{ [\beta]}+\overline{[\alpha]}\cdot\overline{[\gamma]}} \imath[\beta]\cdot d^-(\imath[\gamma]\cdot\imath[\alpha]))\\
&=0,
\end{align*}
as the multiplication in $\Aa^\ast$ is graded commutative.
\end{proof}

Since in characteristic zero the Harrison cohomology of a DGCA injects into the Hochschild cohomology \cite{Barr1968}, we have this immediate corollary of Theorems \ref{thm:obstruction} and \ref{thm:detects}.

\begin{corollary}\label{main-corollary}
Let $\Aa^\ast$ a $(r-1)$-connected ($r>1)$ Poincar\'e DGCA of degree $n$ of Hodge type  with $n\leq 5r-3$. The $Comm_\infty$-quasi-isomorphism class of $\Aa^\ast$ is completely encoded into the DGCA $(H^\ast(\Aa^\ast),\cdot)$ and in the Harrison cohomology class $[\mu_3^{\Aa^\ast}] \in \mathsf{Harr}^{2}(H^\ast(\Aa^\ast),H^\ast(\Aa^\ast))$. In particular, the Harrison class $[\mu_3^{\Aa^\ast}]$ is the only obstruction to the formality of $\Aa^\ast$. 
\end{corollary} 

\begin{remark}[The Bianchi-Massey tensor]\label{rem:compareCN}
Via the spectral sequence relating Harrison to Andr\'e-Quillen cohomology  \cite{Barr1968, Loday1992}, the Harrison cohomology class $[\mu_3^{\Aa^\ast}]\in \mathsf{Harr}^{3,-1}(H^\ast(\Aa^\ast),H^\ast(\Aa^\ast))$ is equally represented by an element $\beta^{\Aa^\ast}\in \mathrm{Hom}^{-1}(S^2(SH^\ast(\Aa^\ast)),H^\ast(\Aa^\ast))$. This latter element is the Bianchi-Massey tensor introduced in \cite{CN} and therein identified as the only obstruction to the formality of a $r$-connected ($r>1)$ Poincar\'e DGCA $\Aa^\ast$ of degree $n$ with $n\leq 5r-3$, as well as a complete invariant of the quasi-isomorphism class of such an $\Aa^\ast$, together with the cohomology algebra $H^\ast(\Aa^\ast)$. In other words, Corollary 
\ref{main-corollary} precisely reproduces the main results from \cite{CN}, and could be obtained from them by translating the Bianchi-Massey tensor into the corresponding Harrison cohomology class. It should however be remarked that our derivation of Corollary 
\ref{main-corollary} is independent of \cite{CN} and can therefore be read as a confirmation of their results. 
\end{remark}

From a computational point of view, it seems that the vanishing of the Bianchi-Massey tensor is easier to be checked than the vanishing of the corresponding Harrison cohomology class, so that the  \cite{CN} approach appears to be more suitable in investigating formality of a given highly connected manifold. The approach via homotopy transfer and Hochschild/Harrison cohomology, on the other hand, appears to be more fitting for the investigation of Massey products. We give an example in Corollaries \ref{cor:Massey-b1}-\ref{cor:Massey-b2} below.

\begin{corollary} \label{cor:Massey-b1}
Let $\Aa^\ast$ be a simply connected Poincar\'e DGCA of degree $n$ 
of Hodge type. 
	\begin{enumerate}
		\item
		If $n \leq 6$, then the Massey product of any Massey triple\footnote{The definition of a Massey $l$-ple in a DGCA $\Aa^\ast$ is recalled in Definition \ref{def:Massey tuple}.} vanishes. 
		\item
		If $n = 7$, then the Massey product vanishes for every Massey triple of degree $\neq (2,2,2)$.
		\item
		If $n = 8$, then the Massey product vanishes for every Massey triple of degree $\notin \{(2,2,2), (2,2,3), (2,3,2), (3,2,2)\}$.
	\end{enumerate}
\end{corollary}
\begin{proof}
The statement is trivial if $n\leq 6$ as in this case $\Aa^\ast$ is formal by Corollary \ref{cor:formal-lowdim}. When the degree of $\Aa^\ast$ is 7, we show that $\mu_3([\alpha],[\beta],[\gamma])$ is possibly nonzero only when $\deg([\alpha],[\beta],[\gamma])=(2,2,2)$. The conclusion then follows by Remark \ref{rem:Massey}. By Corollary \ref{cor:towards-hoch}, we are reduced to show that $\mu_3(\alpha,\beta,\gamma)=0$ for any three harmonic elements $\alpha,\beta,\gamma$ with $\deg(\alpha,\beta,\gamma)\neq(2,2,2)$. By (\ref{eq:recursion5}) we see that for a degree-7 algebra $\Aa^\ast$ the operator $d^-\colon\Qq^{k}_{\mathrm{small}} \to \Qq^{k-1}_{\mathrm{small}}$ is identically zero unless $k=4$.  From the explicit formula for $m_3$ given in Remark \ref{rem:expilicy-m3}, we therefore see that $\mu_3(\alpha,\beta,\gamma)=0$ unless $\deg\alpha+\deg\beta=4$ or $\deg\beta+\deg\gamma=4$. As $\Aa^\ast$ is simply connected, and by Remark \ref{rem:1-gives-0}, $\mu_3(\alpha,\beta,\gamma)$ is zero if one of the elements $\alpha,\beta,\gamma$ has degree 0 or 1. So the only possible nonvanishing triple products come form triples with degrees $(2,2,k)$ or $(k,2,2)$. If $(\alpha,\beta,\gamma)$ is such a triple, then $\mu_3(\alpha,\beta,\gamma)\in \Hh^{k+3}$. So having a possibly nonzero result restricts to the two possibilities $k=2$ or $k=4$. To rule out the $k=4$ possibility, notice that if $\deg(\alpha,\beta,\gamma)=(2,2,4)$, then 
\[
\int m_3(\alpha,\beta,\gamma)=\int \pi_{\Hh^\ast}(d^-(\alpha\cdot \beta)\cdot\gamma) = \int d^-(\alpha\cdot \beta)\cdot\gamma=\langle d^-(\alpha\cdot \beta),\gamma\rangle= 0,
\]
by the orthogonality relation (\ref{def:Hodge-DGA2}). As $\Aa^\ast$ is connected, $\int\colon \Hh^7\to \mathbb{R}$ is an isomorphism and so $m_3(\alpha,\beta,\gamma)=0$. The same argument shows that $m_3(\alpha,\beta,\gamma)=0$ if $\deg(\alpha,\beta,\gamma)=(4,2,2)$.
The proof for the degree-8 case is completely analogous.
\end{proof}
The above corollary can be easily generalized to non-simply connected algebras of a simple kind.
\begin{corollary} \label{cor:Massey-b2}
Let $\Aa^\ast$ be a connected Poincar\'e DGCA of degree $n$. Assume $\Aa^\ast=\Aa^\ast_{\mathrm{sc}}[t_1,\dots,t_k]$, 
with $\Aa^\ast_{\mathrm{sc}}$ a simply connected Poincar\'e DGCA of degree $(n-k)$ of Hodge type, 
where $t_1,\dots,t_k$ are variables in degree 1 with $dt_i=0$.
	\begin{enumerate}
		\item
		If $n \leq 6$, then the Massey product of any Massey triple vanishes.
		\item
		If $n = 7$, then the Massey product vanishes for every Massey triple of degree $\neq (2,2,2)$.
		\item
		If $n = 8$, then the Massey product vanishes for every Massey triple of degree $\notin \{(2,2,2), (2,2,3), (2,3,2), (3,2,2)\}$.
	\end{enumerate}
\end{corollary}
\begin{proof} For $k=0$ we are back to Corollary \ref{cor:Massey-b1}, so let us assume $k\geq 1$. If $n\leq 7$ or if $n=8$ and $k\geq 2$, then $\Aa_{\mathrm{sc}}^\ast$ is formal by Corollary \ref{cor:formal-lowdim}, and so also $\Aa^\ast$ is formal. 
Namely, as $\Aa^\ast_{sc}$ is formal we have a zig-zag of quasi-isomorphisms of DGCAs $\Aa^\ast_{sc}\leftarrow \Dd^\ast_1\to \Dd^\ast_2\leftarrow\cdots \rightarrow H^\ast(\Aa^\ast_{sc})$. As the variables $t_i$ are closed, this induces a zig-zag of quasi-isomorphisms of DGCAs $\Aa^\ast_{sc}[t_1,\dots,t_k]\leftarrow \Dd^\ast_1[t_1,\dots,t_k]\to \Dd^\ast_2[t_1,\dots,t_k]\leftarrow\cdots \rightarrow H^\ast(\Aa^\ast_{sc})[t_1,\dots,t_k]=H^\ast(\Aa^\ast_{sc}[t_1,\dots,t_k])$.
So we are reduced to considering the case $\Aa^\ast=\Aa^\ast_{\mathrm{sc}}[t]$ with $\Aa^\ast_{\mathrm{sc}}$ a degree-$7$ simply connected 
 Poincar\'e DGCA. As $dt=0$, we have $H^\ast(\Aa^\ast)=H^\ast(\Aa^\ast_{\mathrm{sc}})[t]$ and the $A_\infty$-algebra structure on $H^\ast(\Aa^\ast)$ induced by the Kadeishvili-Merkulov construction is just the $\mathbb{F}[t]$-linear extension of the $A_\infty$-algebra structure on $H^\ast(\Aa^\ast_{\mathrm{sc}})$. In particular, it is a $A_3$-algebra structure and 
\begin{align*}
\mu_3([\alpha],[\beta],[\gamma])=
\mu_3^{\mathrm{sc}}&([\alpha_0],[\beta_0],[\gamma_0])\\
&+(-1)^{\deg [\beta]\cdot[\gamma]} \mu_3^{\mathrm{sc}}([\alpha_{-1}],[\beta_0],[\gamma_0])t\\
&+(-1)^{\deg [\gamma]}\mu_3^{\mathrm{sc}}([\alpha_0],[\beta_{-1}],[\gamma_0])t\\
&+\mu_3^{\mathrm{sc}}([\alpha_0],[\beta_0],[\gamma_{-1}])t,
\end{align*}
where we have written $[x]=[x_0]+[x_{-1}]t$ for an homogeneous element $[x]$ in $H^\ast(\Aa^\ast)$.  As $\deg [x_0]=\deg[x]$ and $\deg[x_{-1}]=\deg[x]-1$, by Corollary \ref{cor:Massey-b1} we see that $\mu_3([\alpha],[\beta],[\gamma])$ will vanish unless $\deg([\alpha],[\beta],[\gamma])\notin \{(2,2,2), (2,2,3), (2,3,2), (3,2,2)\}$.
\end{proof}

\section{Splitting off parallel $1$-forms}\label{sec:h1form}

In this section  we    generalize the Cheeger-Gromoll splitting theorem  to the class of  closed  Riemannian  manifolds  whose  harmonic 1-forms are parallel (Theorem \ref{thm:Cheeger-Gromoll}, Remark \ref{rem:CG}).  Using  this  result  we extend   results  in the previous
section that concern  Massey triple products to  the class  of closed orientable    Riemannian  manifolds satisfying the condition of Theorem \ref{thm:Cheeger-Gromoll} (Corollary \ref{cor:Massey-lowdim})  and derive  the almost formality  of compact
 $G_2$-manifolds, obtained   by Chan-Karigiannis-Tsang in \cite{CKT}, from our results (Remark \ref{rem:CKT}).

\begin{theorem} (Cheeger-Gromoll-type splitting) \label{thm:Cheeger-Gromoll}
	Let $(M^n, g)$ be a closed Riemannian manifold, and suppose that all harmonic $1$-forms are parallel. Then there is a Riemannian fibration
	\[
	F^{n-k} \hookrightarrow M^n \longrightarrow T^k,
	\]
	over a flat $k$-torus, where $k = b^1(M^n)$, and $F^{n-k}$ is closed.
	 In fact, $M^n$ can be written as
	\begin{equation} \label{eq:fiber-M}
	M^n = (\R^k \times F^{n-k})/\Lambda,
	\end{equation}
	where $\Lambda \cong \Z^k$ acts isometrically on $F^{n-k}$ and by translations on $\R^k$. Moreover, $M^n$ is diffeomorphic (but not necessarily isometric) to a quotient of $(T^k \times F^{n-k})/\Gamma$, where $\Gamma \subseteq \Transl(T^k) \times \Diffeo(F^{n-k})$ is a finite abelian group.
\end{theorem}

\begin{proof}
	Let $\eps^1, \ldots, \eps^k$ be the parallel $1$-forms with dual vector fields $X_i := (\eps^i)^\#$. Since the distribution spanned by the $(X_i)$ is parallel, so is its orthogonal complement, and the de Rham splitting theorem implies that the universal cover of $(M, g)$ is a Riemannian product
	\[
	\tilde M^n = \R^k \times \tilde F^{n-k},
	\]
	and the lifts $\tilde X_i$ of $X_i$ must be tangent to $\R^k$ and parallel, whence constant. As the deck transformations of the covering $\tilde M^n \to M$ must preserve $\tilde X_i$, they must act by translation on $\R^k$, whence the deck group $\Pi \cong \pi_1(M^n)$ must be contained in
	\[
	\Pi \subseteq \Transl_{\R^k} \times \Isom(\tilde F^{n-k}).
	\]
	Let $\Pi_0 \subseteq \Pi$ be the normal subgroup which acts trivially on $\R^k$, and let $F^{n-k} := \tilde F^{n-k}/\Pi_0$. Then we may write
	\[
	(M^n, g) = (\R^k \times \tilde F^{n-k})/\Pi =  (\R^k \times F^{n-k})/\Lambda,
	\]
	with a faithful map $\Lambda := \Pi/\Pi_0 \hookrightarrow \Transl_{\R^k}$, so that $\Lambda$ is abelian and torsion free, i.e., $\Lambda \cong \Z^l$ for some $l$. Therefore, $\Pi_0$ contains the commutator group $[\Pi, \Pi]$, so that there is a surjection $H_1(M^n) = \Pi/[\Pi, \Pi] \to \Lambda$, and as $H_1(M^n)/Tor(H_1(M^n)) \cong \Z^k$ this means that $k \geq l$.
	
	On the other hand, projection onto the first factor of $\R^k \times F^{n-k}$ implies that $\R^k/\Lambda$ is compact, whence $l = k$, and the image of the inclusion $\Lambda \hookrightarrow \Transl_{\R^k} \cong \R^k$ must be a full lattice, so that $\R^k/\Lambda$ is a flat $k$-torus, and this projection yields the asserted fiber bundle structure.
	
	Consider the homomorphism $\rho: \Lambda \cong \Z^k \to \Isom(F^{n-k})$ from (\ref{eq:fiber-M}). The closure $G := \overline{\rho(\Lambda)} \subseteq \Isom(F^{n-k})$ is a compact abelian Lie group, as $\Isom(F^{n-k})$ is a compact Lie group \cite{MS39}. In particular, $G$ has finitely many components. Thus, $\Lambda_0 := \rho^{-1}(G_0) \subseteq \Lambda$ is a subgroup of finite index, where $G_0 \subseteq G$ denotes the identity component, and we have a Riemannian covering map
	\[
	\hat M^n := (\R^k \times F^{n-k})/\Lambda_0 \longrightarrow M^n = (\R^k \times F^{n-k})/\Lambda
	\]
	which is the quotient by a free action of the finite abelian group $\Gamma := \Lambda/\Lambda_0$ on $\hat M^n$.
	
	We finally have to show that $\hat M^n$ is diffeomorphic to $T^k \times F^{n-k}$. As $G_0$ is a compact connected abelian Lie group, 
it is isomorphic to the torus. 
Then $\rho|_{\Lambda_0}: \Lambda_0 \rightarrow G_0$ extends to 
a Lie group homomorphism $\bar \rho: \R^k \cong \Lambda_0 \otimes \R \rightarrow G_0$. 
Then 
$$
T^k \times F^{n-k} \cong 
(\R^k/ \Lambda_0) \times F^{n-k} \rightarrow (\R^k \times F)/\Lambda_0, \qquad 
([v],f) \mapsto [(v, \bar \rho (v) f)]
$$
gives a diffeomorphism. 
\end{proof}

\begin{remark}\label{rem:CG}
By Bochner's theorem \cite{Bochner1946}, all harmonic 1-forms on a Riemannian manifold with nonnegative Ricci curvature are parallel, so that Theorem \ref{thm:Cheeger-Gromoll} applies. In fact, the conclusion of Theorem \ref{thm:Cheeger-Gromoll}  in the case of nonnegative Ricci curvature is also known as the Cheeger-Gromoll   splitting theorem \cite{CG1971, CG1972}. This holds, in particular, for $G_2$- and $\text{Spin}(7)$-manifolds as these are Ricci flat.
\end{remark}

\begin{proposition} \label{prop:reduce-b1}
Let $(M^n, g)$ be as in Theorem \ref{thm:Cheeger-Gromoll}, and orientable. Then $\Om^\ast(M)$ is equivalent to $(\Lambda^\ast \Hh^1(M^n)) \otimes \Qq^\ast = \Qq^\ast[t_1, \ldots, t_k]$ 
with $k = b^1(M^n)$, where $\Qq^\ast$ is a simply connected non-degenerate Poincar\'e DGCA of Hodge type of degree $(n-k)$.
\end{proposition}

\begin{proof}
By Theorem \ref{thm:Cheeger-Gromoll}, $M^n$ is diffeomorphic to $(T^k \times F^{n-k})/\Gamma$, where $\Gamma \subseteq \Transl(T^k) \times \Diffeo(F^{n-k})$ is a finite abelian group. Let $\Gamma' \subseteq \Diffeo(F^{n-k})$ be the projection of $\Gamma$ to the second factor, and choose a product metric $\tilde g:= g_0 \oplus g_1$ on $T^k \times F^{n-k}$ such that $g_0$ is flat and $g_1$ is $\Gamma'$-invariant. Then $\Gamma$ acts by isometries so that $\pi: (T^k \times F^{n-k}, \tilde g) \to (M^n, g)$ is a Riemannian covering for some metric $g$ on $M$. Clearly, the translation invariant fields on $T^k$ induce parallel vector fields on $M^n$, so that $(M^n, g)$ has $k = b^1(M^n)$ parallel vector fields. Thus, all harmonic $1$-forms on $(M^n,g)$ are parallel.

As $\Gamma$ is a finite group, there is a projection map
\[
\pi_\Gamma: \Om^\ast(T^k \times F^{n-k}) \longrightarrow \Om^\ast(T^k \times F^{n-k})^\Gamma, \qquad \alpha \longmapsto \dfrac1{|\Gamma|} \sum_{\gamma \in \Gamma} \gamma^\ast(\alpha),
\]
and since $\Gamma$ acts by isometries, it follows that $\pi_\Gamma$ commutes with $d, d^\ast$ and $\triangle$. In particular, applying $\pi_\Gamma$ to the components of the Hodge decomposition of $\Om^\ast(T^k \times F^{n-k})$ induces a Hodge decomposition
\begin{equation} \label{eq:Hodge-Gamma}
\begin{array}{ll}
\Om^\ast(T^k \times F^{n-k})^\Gamma = & d\Om^{\ast-1}(T^k \times F^{n-k})^\Gamma \oplus \Hh^\ast(T^k \times F^{n-k})^\Gamma\\
& \oplus \; d^\ast\Om^{\ast+1}(T^k \times F^{n-k})^\Gamma.
\end{array}
\end{equation}
In the same way, we obtain the Hodge decomposition of $\Qq^\ast := \Om(F^{n-k})^{\Gamma'}$
\begin{equation} \label{eq:Hodge-Q}
\Qq^\ast = d\Qq^{\ast-1} \oplus \Hh^\ast(F^{n-k})^{\Gamma'} \oplus d^\ast\Qq^{\ast+1}.
\end{equation}
Furthermore, since all $\alpha^k \in \Lambda^\ast \Hh^1(T^k)$ are parallel, we have for $\beta^l \in \Qq^\ast$: $\ast(\alpha^k \wedge \beta^l) = \pm (\ast_{T^k} \alpha^k) \wedge (\ast_{F^{n-k}} \beta^l) \in \Aa^\ast$, whence $d(\alpha^k \wedge \beta^l) = \pm \alpha^k \wedge d\beta^l$ and $d^\ast(\alpha^k \wedge \beta^l) = \pm \alpha^k \wedge d^\ast \beta^l$. Thus, (\ref{eq:Hodge-Q}) induces a Hodge decomposition of $\Aa^\ast := \Lambda^\ast \Hh^1(T^k) \wedge \Qq^\ast \cong \Qq^\ast[t_1, \ldots, t_k]$
\begin{equation} \label{eq:Hodge-productgamma}
\Aa^\ast = d\Aa^{\ast-1} \oplus (\Lambda^\ast \Hh^1(T^k) \wedge \Hh^\ast(F^{n-k})^{\Gamma'}) \oplus d^\ast\Aa^{\ast+1},
\end{equation}
and the $\Gamma$-invariant harmonic forms of the product metric on $T^k \times F^{n-k}$ decompose as
\[
\Hh^\ast(T^k \times F^{n-k})^\Gamma = (\Hh^\ast(T^k) \wedge \Hh^\ast(F^{n-k}))^\Gamma = \Lambda^\ast \Hh^1(T^k) \wedge \Hh^\ast(F^{n-k})^{\Gamma'},
\]
which thus coincides with the space of harmonic forms in $\Aa^\ast$ by (\ref{eq:Hodge-productgamma}). Therefore, the inclusion $\Aa^\ast \hookrightarrow \Om^\ast(T^k \times F^{n-k})^\Gamma$ is a quasi-isomorphism, and $\pi^\ast: \Om^\ast(M^n) \to \Om^\ast(T^k \times F^{n-k})^\Gamma$ is an isomorphism, whence $\Om^\ast(M^n)$ is quasi-isomorphic to $\Aa^\ast$ as claimed.

By the K\"unneth formula, $\dim H^1(\Aa^\ast) = k + \dim H^1(\Qq^\ast)$, and as $\dim H^1(\Aa^\ast) = b^1(M^n) = k$, it follows that $H^1(\Qq^\ast) = 0$, i.e., $\Qq^\ast$ is a simply connected $(n-k)$-dimensional Poincar\'e algebra as asserted.
\end{proof}

If $n-b^1(M) \leq 6$, then $\Qq^\ast$ in Proposition \ref{prop:reduce-b1} is formal by Corollary \ref{cor:formal-lowdim}, whence so is $\Qq^\ast[t_1, \ldots, t_k]$ by the same argument as in the proof of Corollary \ref{cor:Massey-b2}. Since this in turn is equivalent to $\Om^\ast(M^n)$ by Proposition \ref{prop:reduce-b1}, we immediately obtain the following.

\begin{corollary} \label{cor:lowdim-formal}
Let $(M^n, g)$ be a closed orientable Riemannian manifold, and suppose that all harmonic $1$-forms are parallel. If $n - b^1(M^n) \leq 6$, then $M^n$ is formal. In particular, any $G_2$-manifold with $b^1(M) > 0$ (i.e., with holonomy strictly contained in $G_2$) and any $\text{Spin(7)}$-manifold with $b^1(M) > 1$ is formal.
\end{corollary}

\begin{corollary} \label{cor:Massey-lowdim}
Let $(M^n, g)$ be a closed oriented $n$-dimensional Riemannian manifold such that all harmonic $1$-forms are parallel.  Consider    cohomology   group  $H^\ast(M, \R)$.
	\begin{enumerate}
		\item
		If $n \leq 6$, then the Massey product of any Massey triple vanishes.
		\item
		If $n = 7$, then the Massey product vanishes for every Massey triple of degree $\neq (2,2,2)$.
		\item
		If $n = 8$, then the Massey product vanishes for every Massey triple of degree $\notin \{(2,2,2), (2,2,3), (2,3,2), (3,2,2)\}$.
	\end{enumerate}
\end{corollary}

\begin{proof}    Assume that $(M^n, g)$ satisfies the condition  of Corollary \ref{cor:Massey-lowdim}. By Proposition \ref{prop:reduce-b1} the DGCA $\Om^\ast(M)$ is equivalent to $\Qq^\ast[t_1, \ldots, t_k]$, $k = b^1(M)$, for a simply connected Poincar\'e DGCA $\Qq^\ast$ of Hodge type. Then apply Corollary \ref{cor:Massey-b2}.
\end{proof}

\begin{remark}\label{rem:CKT}
Corollary \ref{cor:Massey-lowdim} covers $G_2$- and $\text{Spin}(7)$-manifolds 
as these are Ricci-flat and hence have parallel harmonic $1$-forms. 
Corollary \ref{cor:Massey-lowdim} (2) for $G_2$-manifolds 
was first proved by \cite[Theorem 4.15]{CKT}. 
\end{remark}


\appendix

\section{Massey Products}\label{appendix.massey}

Let $\Aa^*$ be a DGCA and set $H^*=H^*(\Aa^*)$. 
For $x_1, \cdots, x_l \in H^*$, 
define the {\it $l$th order Massey product} 
$\la x_1, \cdots x_l \ra$, which is a subset of $H^{\sum_{i=1}^l |x_i| +2-l}$, 
for $l \geq 3$ as follows.

\begin{definition}
A set $\{ a_{i j} \}_{1 \leq i \leq j \leq l, (i,j) \neq (1,l)} \subseteq \Aa^*$ is called 
a {\it defining system} for 
the $l$th order Massey product $\la x_1, \cdots x_l \ra$ if 
\begin{itemize}
\item for $1 \leq i \leq l$, $a_{i i}$ represents $x_i$, and 
\item for $1 \leq i < j \leq l$ with $(i,j) \neq (1,l)$, 
\begin{align}\label{eq:defs}
d (a_{i j}) = \sum_{k=i}^{j-1} \bar{a}_{i k} a_{k+1, j}, 
\end{align}
\end{itemize}
where $\bar{a}_{i k} = (-1)^{|a_{i k}|} a_{i k}.$ 
Then we have the cocycle 
$$
\alpha_{\{ a_{i j} \}} = \sum_{k=1}^{l-1} \bar{a}_{1 k} a_{k+1, l}. 
$$
The {\it $l$th order Massey product} $\la x_1, \cdots x_l \ra$ is defined by 
$$
\la x_1, \cdots x_l \ra = 
\left\{ \left[ \alpha_{\{ a_{i j} \}} \right] \mid \{ a_{i j} \} \mbox{ is a defining system} \right\} 
\subseteq H^{\sum_{i=1}^l |x_i| +2-l}. 
$$

The Massey product $\la x_1, \cdots, x_l \ra$ is said to be {\it trivial} 
if it contains $0$.
\end{definition}

\begin{definition}\label{def:Massey tuple}
A {\em Massey $l$-tuple (of $\Aa^\ast$) }is an $l$-tuple $(x_1, \cdots, x_l)$ of cohomology classes 
such that $\la x_1, \cdots, x_l \ra \neq \emptyset.$ 
We call $(|x_1|, \cdots, |x_l|)$ the {\em degree of $(x_1, \cdots, x_l)$}.
\end{definition}

\begin{lemma} \label{lem:vanish MP}
Let $(x_1, \cdots, x_l)$ be a Massey $l$-tuple for $l \geq 3$. 
Then the $l$th order Massey product $\la x_1, \cdots, x_l \ra$ is trivial 
if one of the following holds. 
\begin{enumerate}
\item
We have $x_p=0$ for some $p$. 
\item 
A DGCA $\Aa^*$ is connected and $|x_p|=0$ for some $p$. 
\item 
A DGCA $\Aa^*$ is a connected Poincar\'e DGCA of degree $n$ and 
$\sum_{i=1}^l |x_i| +2-l =n$. 
\end{enumerate}
\end{lemma}

\begin{proof}
By\cite[(2.3)]{Kraines1966}, we have 
$
c \la x_1, \cdots, x_l \ra \subset \la x_1, \cdots, c x_p, \cdots, x_l \ra
$
for any $c \in \F$. Then (1) is immediate. 

Next, we show (2). 
As $\Aa^*$ is connected, 
the product in cohomology $H^0(\Aa^\ast)\otimes H^k(\Aa^\ast)\to H^k(\Aa^\ast)$ 
is the multiplication by scalars on $H^k(\Aa^\ast)$ by Remark \ref{rem:poincare}. 
We also have $x_{p-1} \cdot x_p = x_p \cdot x_{p+1}=0$ by (\ref{eq:defs}). 
Then we see that $x_p=0$ or $x_{p-1}=x_{p+1}=0$. 
Hence (2) follows from (1). 

Finally, we prove (3). 
Fix a defining system $\{ a_{i j} \}$. 
As $a_{2 l}$ does not appear in 
the right hand side of (\ref{eq:defs}), 
$\{ a_{i j} + \delta_{i 2} \delta_{j l} \gamma \}$ for any cocycle $\gamma \in \Aa^{|a_{2 l}|}$ 
is also a defining system. 
As $[a_{1 1}] = x_1$, we see that 
$$
\la x_1, \cdots, x_l \ra \supset \alpha_{\{ a_{i j} \}} + x_1 \cdot H^{\sum_{i=2}^l |x_i| +2-l} 
= \alpha_{\{ a_{i j} \}} + x_1 \cdot H^{n-|x_1|}. 
$$
When $x_1=0$, $\la x_1, \cdots, x_l \ra$ is trivial by (1). 
When $x_1 \neq 0$, 
we have $x_1 \cdot H^{n-|x_1|} \neq 0$ by the non-degeneracy of $\langle - , - \rangle$ 
defined in Definition \ref{def:Poincare} (1). 
This implies that $x_1 \cdot H^{n-|x_1|} = H^n$ as $\dim H^n=1$, 
and hence $\la x_1, \cdots, x_l \ra = H^n \ni 0$. 
\end{proof}

\begin{proposition} \label{prop:trivial MP}
Let $\Aa^\ast$ be a $(r-1)$-connected $(r>1)$ Poincar\'e DGCA of degree $n$. 
For $l \geq 3$, the following holds. 
\begin{enumerate}
\item
If $n \leq r(l+1) +(1-l)$, 
the $l'$ th order Massey product of any Massey $l'$-tuple is trivial for any $l' \geq l$. 
\item
If $n = r(l+1) +(2-l)+q$ for $q \geq 0$, 
then the $l$th order Massey product is trivial for every Massey $l$-tuple 
of degree $\not\in \{ (r+a_1, \cdots, r+a_l) \mid a_i \geq 0,\ 0 \leq \sum_{i=1}^l a_i \leq q\}$. 
\end{enumerate}  
\end{proposition}

\begin{proof}
Let $(x_1,\cdots.x_l)$ be a Massey $l$-tuple. 
By Lemma \ref{lem:vanish MP}, 
we must have 
\begin{align}\label{eq:nontriv MP}
|x_i| \geq r \quad \mbox{ for any } i, \qquad \mbox{and}\qquad  
\sum_{i=1}^l |x_i| +2-l \leq n-r
\end{align}
so that 
$\la x_1, \cdots, x_l \ra$ is nontrivial 
because $H^k=0$ for $1 \leq k \leq r-1$, $n-r+1 \leq k \leq n-1$ and $k \geq n+1$. 
In particular, we must have 
$$
rl + 2-l \leq n-r \qquad \Longleftrightarrow \qquad n \geq r(l+1) +(2-l), 
$$
which proves (1). 
Next, we prove (2). 
By the first equation of (\ref{eq:nontriv MP}), 
we must have $|x_i| = r+a_i$ for $a_i \geq 0$. 
Then the second equation becomes 
$$
r l + \sum_{i=1}^l a_i +2-l \leq rl + (2-l)+q \qquad \Longleftrightarrow \qquad 
0 \leq \sum_{i=1}^l a_i \leq q, 
$$
which proves (2). 
\end{proof}

The following is immediate from Proposition \ref{prop:trivial MP}. 

\begin{corollary}
Let $\Aa^\ast$ be a $(r-1)$-connected $(r>1)$ Poincar\'e DGCA of degree $n$. 
If $n \leq 4r-2$, all of the Massey products are trivial. 
If $n \leq 5r-3$, all of the $l$ th order Massey products are trivial for $l \geq 4$.
\end{corollary}

Note that we obtain Corollary \ref{cor:Massey-b1} also from Proposition \ref{prop:trivial MP}.

\

\subsection*{Acknowledgement} 
HVL and LS  thank  the   VIASM in Hanoi  and  MPIMIS  in Leipzig   for 
hospitality and   excellent working  conditions  during  their   visit in 2018,  where  a part  of this project    was discussed.
KK thanks Hisashi Kasuya for explaining Sullivan's work. 
The authors wish to thank Diarmuid Crowley and Johannes Nordstr\"om for discussions on \cite{CN}   and Pavel Hajek for   helpful   comments  on the relation of this article to  \cite{Hajek2018}, \cite[Theorem 1.1]{LS2007} and \cite[Lemma 11.1]{CFL2015}.


\end{document}